\documentclass[a4paper,11pt,english]{article}

\usepackage{a4wide,bm,epsfig,booktabs}

\usepackage{amsmath}
\usepackage{tikz}
\usepackage{setspace}
\usepackage{relsize}  
\usepackage{graphicx}  
\usepackage{cancel} 
\usepackage{slashed}
\usepackage{verbatim} 
\usepackage{enumerate} 
\usepackage{stmaryrd} 
\usepackage{cite}
\usepackage[flushmargin,hang]{footmisc}
\usepackage{fancyhdr} 
\usepackage[arrow, matrix,curve]{xy}
\usepackage{hyperref}
\usepackage{amssymb}
\usepackage{amsthm}
\usepackage{eucal}
\usepackage{mathtools}
\usepackage{scalerel}

\DeclareMathOperator*{\innt}{\ThisStyle{\vstretch{0.9}{\hstretch{1.5}{\rotatebox{10}{$\SavedStyle\hspace{-0.5pt}\!\int\!\hspace{-0.5pt}$}}}}}
\DeclareMathOperator*{\Der}{\ThisStyle{\hstretch{1.2}{\rotatebox{0}{$\SavedStyle\delta^r$}}}}

\newcommand{\AI} {\mathfrak{X}}
\newcommand{\slim}[1] {\rightarrow_{#1}}

\newcommand{\NN} {\mathbb{N}}

\newcommand{\RR} {\mathbb{R}}

\newcommand{\lip}  {\mathrm{lip}}

\newcommand{\kk} {\mathrm{k}}
\newcommand{\DE} {\mathrm{D}}

\newcommand{\bl} {\boldsymbol{[}}
\newcommand{\br} {\boldsymbol{]}}

\newcommand{\uu} {\mathfrak{u}}
\newcommand{\mm} {\mathfrak{m}}

\newcommand{\oo} {\mathfrak{o}}
\newcommand{\pp} {\mathfrak{p}}
\newcommand{\qq} {\mathfrak{q}}

\newcommand{\vv} {\mathfrak{v}}
\newcommand{\ww} {\mathfrak{w}}

\newcommand{\mmm} {\mm}

\newcommand{\ppp} {\pp}
\newcommand{\qqq} {\qq}

\newcommand{\vvv} {\vv}
\newcommand{\www} {\ww}

\newcommand{\inverse}[1] {{\mathfrak{inv}(#1)}}
\newcommand{\inver} {\mathfrak{inv}}

\newcommand{\psii}  {\psi}

\newcommand{\euler} {\mathrm{e}}

\newcommand{\SEM} {\mathfrak{P}}

\newcommand{\SEMM} {\mathfrak{Q}}

\newcommand{\inv} {\mathrm{inv}}
\newcommand{\DIDE} {\mathfrak{D}}
\newcommand{\EV} {\mathrm{Evol}}
\newcommand{\evol} {{\mathrm{evol}}}

\newcommand{\pl} {{\boldsymbol{+}}}

\newcommand{\elll} {u}

\newcommand{\bchi}[1] {\boldsymbol{\chi}_{#1}}
\newcommand{\bbchi} {\boldsymbol{\chi}}

\newcommand{\B} {\mathrm{B}}
\newcommand{\OB} {\ovl{\B}}

\newcommand{\chart} {\Xi}
\newcommand{\chartinv} {\Xi^{-1}}
\newcommand{\RT} {\mathrm{R}}

\newcommand{\Ad} {\mathrm{Ad}}
\newcommand{\Add} {\mathbf{Ad}}

\newcommand{\com}[1] {\llbracket{#1}\rrbracket}

\newcommand{\ul}[1] {\underline{#1}}

\newcommand{\CP} {\mathrm{CP}}

\newcommand{\DP} {\DIDE\mathrm{P}}

\newcommand{\dind}  {\mathrm{s}}

\newcommand{\dindq}  {\mathrm{q}}

\newcommand{\mackeyconst} {\mathfrak{c}}

\newcommand{\mackeyindex} {\mathfrak{l}}

\newcommand{\comp}[1] {\ovl{#1}}

\newcommand{\Lamb} {\boldsymbol{\Lambda}}

\newcommand{\llleq} {\preceq}

\newcommand{\exxp}  {\exp}

\newcommand{\conj}  {\mathbf{c}}

\newcommand{\id} {\mathrm{id}}

\newcommand{\dermapdiff}  {\omega}
\newcommand{\dermapinvdiff}  {\upsilon}

\newcommand{\MA} {\mathcal{A}}
\newcommand{\MAU} {\MA^\times}

\newcommand{\ovl}[1] {\overline{#1}}

\newcommand{\dd} {\mathrm{d}}
\newcommand{\im} {\mathrm{im}}
\newcommand{\dom} {\mathrm{dom}}

\newcommand{\U}  {\mathcal{U}}
\newcommand{\V}  {\mathcal{V}}

\newcommand{\w} {\omega}

\newcommand{\deff} {if and only if } 
\newcommand{\defff} {if } 

\newcommand{\mg} {\mathfrak{g}}
\newcommand{\mh} {\mathfrak{h}}

\newcommand{\cp} {\circ}

\newcommand{\mult}  {\mathrm{m}}

\newcommand{\compact} {\mathrm{C}}
\newcommand{\compacto} {\mathrm{K}}

\newcommand{\he} {\hspace{1pt}}

\renewcommand{\theenumi}{\arabic{enumi})} 
\renewcommand{\labelenumi}{\theenumi}

\let\origenumerate\enumerate
\def\enumerate{\origenumerate\itemsep0pt}
\let\origitemize\itemize
\def\itemize{\origitemize\itemsep0pt}

\newtheorem{innercustomthm}{Theorem}
\newenvironment{customthm}[1]
  {\renewcommand\theinnercustomthm{#1}\innercustomthm}
  {\endinnercustomthm}
  
 \newtheorem{innercustompr}{Proposition}
\newenvironment{custompr}[1]
  {\renewcommand\theinnercustompr{#1}\innercustompr}
  {\endinnercustompr}

\newtheorem{theorem}{Theorem}
\newtheorem{proposition}{Proposition}
\newtheorem{lemma}{Lemma}
\newtheorem{corollary}{Corollary}
\newtheorem{remark}{Remark}

\newtheorem{example}{Example}

\makeatletter
\def\blfootnote{\gdef\@thefnmark{}\@footnotetext}
\makeatother

\begin{document}
\title{The Regularity Problem for Lie Groups with\\ Asymptotic Estimate Lie Algebras}
\author{
  \textbf{Maximilian Hanusch}\thanks{\texttt{mhanusch@math.upb.de}}
  \\[1cm]
  Institut f\"ur Mathematik \\
  Lehrstuhl f\"ur Mathematik X \\
  Universit\"at W\"urzburg \\
  Campus Hubland Nord \\
  Emil-Fischer-Stra\ss e 31 \\
  97074 W\"urzburg \\
  Germany
}
\date{January 24, 2020}
\maketitle

\begin{abstract}  
We solve the regularity problem for Milnor's infinite dimensional Lie groups in the asymptotic estimate context. Specifically, let $G$ be a Lie group with asymptotic estimate Lie algebra $\mg$, and denote its evolution map by $\evol\colon \DE\equiv \dom[\evol]\rightarrow G$, i.e., $\DE\subseteq C^0([0,1],\mg)$. We show that $\evol$ is $C^\infty$-continuous on $\DE\cap C^\infty([0,1],\mg)$ \deff $\evol$ is $C^0$-continuous on $\DE\cap C^0([0,1],\mg)$. We furthermore show  that $G$ is k-confined for $k\in \NN\sqcup\{\lip,\infty\}$ if $G$ is constricted. (The latter condition is slightly less restrictive than to be asymptotic estimate.) Results obtained in a previous paper then imply that an asymptotic estimate Lie group $G$ is $C^\infty$-regular \deff it is Mackey complete, locally $\mu$-convex, and has Mackey complete Lie algebra -- In this case, $G$ is $C^k$-regular for each $k\in \NN_{\geq 1}\sqcup\{\lip,\infty\}$ (with ``smoothness restrictions'' for $k\equiv\lip$), as well as $C^0$-regular if $G$ is even sequentially complete with integral complete Lie algebra. 
\end{abstract}

\tableofcontents 

\section{Introduction} 
In 1983 Milnor introduced his regularity concept \cite{MIL} as a tool to extend proofs of fundamental Lie theoretical facts to infinite dimensions. Specifically, he adapted (and weakened) the regularity concept introduced in 1982 by Omori et al.\ for Fr\'{e}chet Lie groups \cite{OMORI} to such Lie groups that are  modeled over complete Hausdorff locally convex vector spaces. Then, he proved that given Lie groups $G,H$ with Lie algebras $\mg,\mh$ such that $H$ is regular and $G$ is connected and simply connected, then each continuous Lie algebra homomorphism $\mg\rightarrow \mh$ integrates (necessarily unique) to a smooth Lie group homomorphism $G\rightarrow H$. In this paper, we work in the slightly more general setting introduced in \cite{HG} by Gl\"ockner. Specifically, this means that any completeness presumption on the modeling space is dropped.\footnote{To prevent confusion, we additionally remark that Milnor's definition of an infinite dimensional manifold $M$ involves the requirement that $M$ is a regular topological space, i.e., fulfills the separation axioms $T_2, T_3$. Deviating from that, in \cite{HG}, only the $T_2$ property of $M$ is explicitly presumed. This restriction, however, makes no difference in the Lie group case, because topological groups are automatically $T_3$.} 
Roughly speaking, regularity is concerned with definedness and smoothness/continuity of the product integral. This is a notion that naturally generalizes the concept of the Riemann integral for curves in locally convex vector spaces to infinite dimensional Lie groups (Lie algebra valued curves are thus integrated to Lie group elements). For instance, the exponential map of a Lie group is the restriction of the product integral to constant curves; and, given a principal fibre bundle, holonomies are product integrals of such Lie algebra valued curves that are pairings of smooth connections with derivatives of curves in the base manifold of the bundle.   Although individual arguments show that
the generic infinite dimensional Lie group is $C^\infty$-regular or stronger, only recently general regularity criteria had been found \cite{HGGG, KHN2, RGM}.  
In this paper, we solve the regularity problem in the asymptotic estimate context. Specifically, a Lie group $G$ is said to be  asymptotic estimate if its Lie algebra $\mg$ is asymptotic estimate, i.e, if 
to each continuous seminorm $\vv$ on $\mg$, there exists a continuous seminorm $\vv\leq \ww$ on $\mg$ such that
\begin{align}
\label{assaaas}
	\vv(\bl X_1,\bl X_2,\bl \dots,\bl X_n,Y\br{\dots}\br\br\br)\leq \ww(X_1)\cdot{\dots}\cdot \ww(X_n)\cdot \ww(Y)
\end{align}
holds for all $X_1,\dots,X_n,Y\in \mg$ and $n\geq 1$.\footnote{By the best of our knowledge, the term asymptotic estimate had been introduced in \cite{BOSECK} in the context of (not neccessarily associative) Hausdorff locally convex algebras. In \cite{CSP}, a weaker definition has been used that in particular specializes to the condition $(*)$ formulated in \cite{HGIA} for continuous inverse algebras (associativity). Our definition is a priori slightly weaker (a detailed combinatorical analysis involving the Jacobi identity might show equivalence) than the notion in \cite{CSP} when applied to Lie algebras.}  
For instance, abelian Lie groups are asymptotic estimate, and the same is true for Lie groups with nilpotent Lie algebras. Also Banach Lie groups are asymptotic estimate, because their Lie bracket is submultiplicative. Notably, the class of asymptotic estimate Lie algebras has good permanence properties, as it is closed under passage to subalgebras, Hausdorff quotient Lie algebras, as well as closed under arbitrary cartesian products (hence, e.g., under projective limits). Various examples of asymptotic estimate Lie groups are thus obtained by taking, e.g., products of Banach Lie groups with Lie groups with nilpotent Lie algebras. 

The results obtained in this paper are basically due to a deeper analysis of the adjoint equation that (to a certain extent) had been started in \cite{RGM}. More specifically, we prove a certain approximation property of the adjoint action that we then use to show  the following statements:
\begingroup
\setlength{\leftmargini}{17pt}
{
\renewcommand{\theenumi}{{\bf \arabic{enumi})}} 
\renewcommand{\labelenumi}{\theenumi}
\begin{enumerate}
\item
\label{aprop1}
If \eqref{assaaas} holds, then  
$C^\infty$-continuity  
of the evolution map is equivalent to 
$C^0$-continuity.\footnote{Apart from the mentioned approximation property, here we use a generalization of the argument used in the proof of Lemma 16 in \cite{RGM} for the abelian case. We specifically remark that the equivalence of $C^\infty$ and $C^0$-continuity already follows in the abelian case, when Lemma 16 in \cite{RGM} is combined with Theorem 1 in \cite{RGM}.}  
\item
\label{aprop2} 
$G$ is k-confined for $k\in \NN\sqcup\{\lip,\infty\}$ if $\mg$ is constricted. The latter condition is slightly less restrictive than Condition \eqref{assaaas}, and k-confinedness  
is an integrability condition that was introduced in \cite{RGM}.
\end{enumerate}}
\endgroup
\noindent 
In particular, we will prove that, cf.\ Theorem (\ref{confevv111}.\ref{confevv22}
\vspace{6pt}

{\it``Let $G$ be an infinite dimensional Lie group in Milnor's sense with asymptotic estimate Lie bracket. 
Then,   
$G$ is $C^\infty$-regular \deff $G$ is locally $\mu$-convex, Mackey complete, and has Mackey complete Lie algebra -- In this case, $G$ is $C^k$-regular for each $k\in \NN_{\geq 1}\sqcup \{\lip,\infty\}$.''
}
\vspace{6pt}

Here, Mackey completeness of $G$ (cf.\ Sect.\ \ref{uepdspodspodsapoa}) generalizes Mackey completeness as defined for locally convex vector spaces (as, e.g., for $\mg$); and, ``locally $\mu$-convex'' means that to each continuous seminorm $\uu$ on the modeling space $E$ of $G$, and to each chart $\chart\colon G\supseteq \U\rightarrow \V\subseteq E$ of $G$ around $e$ with $\chart(e)=0$, there exists a continuous seminorm $\uu\leq \oo$ on $E$ such that 
\begin{align}
\label{podspodspopodsds}
	(\uu\cp\chart)\big(\chart^{-1}(X_1)\cdot {\dots}\cdot \chart^{-1}(X_n)\big)\leq \oo(X_1)+{\dots}+\oo(X_n)
\end{align} 
holds for all $X_1,\dots,X_n\in E$ with $\oo(X_1)+{\dots}+\oo(X_n) \leq 1$.  
This notion had been introduced in \cite{HGGG} as a tool to investigate regularity properties of weak direct products of Lie groups; and then was shown to be equivalent to $C^0$-continuity of the evolution map in \cite{RGM}. Apart from the regularity problem, local $\mu$-convexity has turned out to be of relevance also for other problems in infinite dimensional Lie theory. For instance, it was shown in \cite{RGM} (confer also Lemma 2 in \cite{TRM}) that local $\mu$-convexity implies continuity of the evolution map  w.r.t.\ the $L^1$-topology that plays a role, e.g., in the measurable regular context \cite{HGM}. Moreover, it was shown in \cite{TRM} that local $\mu$-convexity implies the strong Trotter property \cite{HGM} that is relevant, e.g., in representation theory of infinite dimensional Lie groups \cite{KHN}. The statement \ref{aprop1} proven in this paper, thus in particular extends the range of application of these results to all  asymptotic estimate Lie groups with $C^\infty$-continuous evolution map.    

This paper is organized as follows:
\begingroup
\setlength{\leftmargini}{12pt}
\begin{itemize}
\item
In Sect.\ \ref{sdlkdslkdslkds}, we state the main results obtained in this paper; and provide the solution to the regularity problem in the asymptotic estimate (constricted) context, cf.\ Theorem \ref{confevv111}.  
\item 
In Sect.\ \ref{dsdssd}, we provide the basic definitions; and recall the properties of the core mathematical objects of this paper that are relevant for our discussions in the main text. 
\item
In Sect.\ \ref{sdddsdsdsds}, we prove an approximation property of the adjoint action -- and then derive some estimates from this that will be used in Sect.\ \ref{podspodspodspo} and Sect.\ \ref{podspodspodspo1} to prove our main results.
\item
In Sect.\ \ref{podspodspodspo}, we prove the statement made in \ref{aprop2}, cf.\ Proposition \ref{fdkjfdkjdfkjfdjfdlkjfdl}.
\item
In Sect.\ \ref{podspodspodspo1}, we prove the statement made in \ref{aprop1}, cf.\ Theorem \ref{aoelsaoesalsaoelsa}.
\end{itemize}
\endgroup

\section{Precise Statement of the Results}
\label{sdlkdslkdslkds}
Let $G$ be an infinite dimensional Lie group in the sense of \cite{HG} (cf.\ Definition 3.1 and Definition 3.3 in \cite{HG}) that is modeled over the Hausdorff locally convex vector space $E$, with corresponding system of continuous seminorms $\SEM$. We denote the Lie algebra of $G$ by $(\mg,\bl \cdot,\cdot\br)$, the identity element by $e\in G$, the Lie group  multiplication by $\mult\colon G\times G\rightarrow G$, and define $\RT_g:=\mult(\cdot, g)$ for each $g\in G$.       
We furthermore fix a chart $\chart\colon G\supseteq \U\rightarrow \V\subseteq E$ with $\V$ convex, $e\in \U$, and $\chart(e)=0$; and let $\ppp(X):=(\pp\cp\dd_e\chart)(X)$ for each $\pp\in \SEM$, and $X\in \mg$. 
\vspace{6pt}

\noindent
The right logarithmic derivative is defined by $\Der(\mu)=\dd_\mu\RT_{\mu^{-1}}(\dot\mu)\in C^0(D,\mg)$, for $\mu\in C^1(D,G)$ with $D\subseteq \RR$ an interval. The evolution map is given by\footnote{The elementary properties of the evolution map are recalled in Sect.\ \ref{kjdsjlkdsklskjd}.}
\begin{align*}
	\evol\colon \Der(C^1([0,1],G))\rightarrow G,\qquad \Der(\mu)\mapsto \mu(1)\cdot \mu(0)^{-1}.
\end{align*}
For each $k\in \NN\sqcup\{\lip,\infty\}$, we let 
\begin{align*}
	\evol_\kk:=\evol|_{\dom[\evol]\cap C^k([0,1],\mg)};
\end{align*}
and say that $G$ is $C^k$-semiregular \defff $\dom[\evol_\kk]=C^k([0,1],\mg)$ holds. 
We say that $G$ is $C^k$-regular \defff $G$ is $C^k$-semiregular, such that
\begingroup
\setlength{\leftmargini}{11pt}
\begin{itemize}
\item
For $k\in \NN\sqcup\{\infty\}$:\quad\quad\:\: $\evol_\kk$ is smooth w.r.t.\ the $C^k$-topology,
\item
For $k\equiv\lip$:\hspace{39.5pt}\quad\:\: $\evol_\kk$ is of class $C^1$ w.r.t.\ the $C^0$-topology.\footnote{Confer Remark \ref{sdsddssdsd} for an explanation of the deviating definition in the Lipschitz case. Moreover, confer Sect.\ \ref{ksasahgsahgajdkjfdkjfd} for a precise definition of the spaces $C^k([0,1],\mg)$ for $k\in \NN\sqcup \{\lip,\infty\}$.}
\end{itemize}
\endgroup
\noindent
We recall that $\mg$ is said to be Mackey complete \defff $\int \phi(s)\: \dd s\in \mg$ exists for each $\phi\in C^\lip([0,1],\mg)$; as well as integral complete \defff  $\int \phi(s)\: \dd s\in \mg$ exists for each $\phi\in C^0([0,1],\mg)$. We refer to Theorem 2.14 in \cite{COS} for a summary of the alternative  definitions of Mackey completeness commonly used in the literature.  

\subsection{State of the Art}
\label{lkfdlkfdlkdflkfdlfd}
In \cite{KHN2} it had been clarified that $C^\infty$-regularity implies Mackey completeness of $\mg$. In \cite{HGGG}, this was supplemented by proving that $C^0$-regularity implies integral completeness of $\mg$. It was furthermore shown in \cite{HGGG} that $\evol_\kk$ is smooth for $k\in \NN\sqcup\{\infty\}$ \deff it is of class $C^1$; and several regularity criteria were provided there. Then, in \cite{RGM} it was shown that integral-, and Mackey completeness are actually ``if and only if'' conditions.   
More specifically, let us say that $G$ is k-continuous for $k\in \NN\sqcup\{\lip,\infty\}$ \defff $\evol_\kk$ is $C^k$-continuous ($C^0$-continuous for $k\equiv\lip$) on its domain. 
Then,  Theorem 4 in \cite{RGM} (Corollary 13 in \cite{RGM} for $k\equiv \lip$) states
\begin{customthm}{C}\label{Smthn}
\begingroup
\setlength{\leftmargini}{17pt}
\begin{enumerate}
\item
\label{smthn1}
If $G$ is {\rm 0}-continuous and $C^0$-semiregular, then $G$ is $C^0$-regular \deff $\evol_0$ is differentiable at zero \deff $\mg$ is integral complete.
\item
\label{smthn2}
 If $G$ is {\rm k}-continuous and $C^k$-semiregular for $k\in \NN_{\geq 1}\sqcup\{\lip,\infty\}$, then $G$ is $C^k$-regular \deff $\evol_\kk$ is differentiable at zero \deff 
	$\mg$ is Mackey complete.
\end{enumerate}
\endgroup
\end{customthm}
\vspace{-5pt}
\noindent
This solves the smoothness issue in full generality, and reduces the regularity problem to the following two questions: 
\begingroup
\setlength{\leftmargini}{19pt}
{
\renewcommand{\theenumi}{{\bf \Alph{enumi}})} 
\renewcommand{\labelenumi}{\theenumi}
\begin{enumerate}
\item
\label{aaadffda1}
Under which circumstances is a Lie group k-continuous for some given $k\in \NN\sqcup\{\lip,\infty\}$.
\item
\label{aafddfaa2}
	Under which circumstances is a Lie group $C^k$-semiregular for some given $k\in \NN\sqcup\{\lip,\infty\}$.
\end{enumerate}}
\endgroup
\noindent
In \cite{RGM}, these questions had been answered in the $C^0$-topological setting: 

More specifically, 
 Theorem 1 in \cite{RGM} shows that
\begin{customthm}{A}\label{LMC}
$G$ is {\rm 0}-continuous \deff $G$ is locally $\mu$-convex (i.e., fulfills \eqref{podspodspopodsds}). 
\end{customthm}
Moreover, Theorem 3 in \cite{RGM} states that
\begin{customthm}{B}
\label{confev}
Assume that $G$ is locally $\mu$-convex. Then,
\vspace{-4pt} 
\begingroup
\setlength{\leftmargini}{17pt}
\begin{enumerate}
\item
\label{confev1}
 $G$ is $C^0$-semiregular  if $G$ is sequentially complete and $0$-confined.
\item
\label{confev2}
 $G$ is $C^{k}$-semiregular for $k\in \NN_{\geq 1}\sqcup\{\lip,\infty\}$ \deff $G$ is Mackey complete and {\rm k}-confined.
\end{enumerate}
\endgroup
\end{customthm}
Here, sequentially-, and Mackey completeness generalize sequentially-, and Mackey completeness as defined for locally convex vector spaces; and, k-confinedness, for $k\in \NN\sqcup\{\lip,\infty\}$, is an approximation property for $C^k$-curves (the precise definitions are not relevant at this point -- they are recalled and explained in Sect.\ \ref{uepdspodspodsapoa} and Sect.\ \ref{podspodspodspo}, respectively).
\vspace{6pt}

\noindent
Let now $\com{X}\colon \mg\ni Y\mapsto \bl X,Y\br$ for each $X\in \mg$.   
\begingroup
\setlength{\leftmargini}{12pt}
\begin{itemize}
\item
We say that $\mg$ is {\bf asymptotic estimate} \defff for each $\vv\in \SEM$, there exists $\vv\leq\ww\in \SEM$, such that
\begin{align}
\label{fdjfdlkfdfdlkmcx}
	(\vvv\cp \com{X_1}\cp {\dots}\cp\com{X_n})(Y)\leq \www(X_1)\cdot{\dots}\cdot \www(X_n)\cdot\www(Y)
\end{align}
holds for all $X_1,\dots,X_n,\:Y\in \mg$, and $n\geq 1$. 
\item
We say that $\mg$ is {\bf constricted} \defff for each metrizable compact subset $\compacto\subseteq \mg$, and each $\vv\in \SEM$, there exist $C_\vv\geq 0$ and $\vv\leq \ww\in \SEM$ with
\begin{align}
\label{assssasaasasdsdsds}
	\vvv\cp \com{X_1}\cp {\dots}\cp \com{X_n}\leq  C_\vv^n\cdot \www\qquad\quad\forall\: X_1,\dots,X_n\in \compacto,\:\: n\geq 1.
\end{align}
\end{itemize}
\endgroup
\noindent
We say that $G$ is {\bf asymptotic estimate\slash constricted} \defff $\mg$ is asymptotic estimate\slash constricted.
\begin{remark}
\label{dffd}
\begingroup
\setlength{\leftmargini}{17pt}
\begin{enumerate}
\item
\label{r1}
Evidently, $G$ is constricted if $G$ is asymptotic estimate. 
\item
\label{r2}
Constrictedness as defined above is a weaker condition than constrictedness as defined in \cite{RGM}, because there it was formulated in terms of bounded subsets instead of metrizable compact  ones. We will use the more general definition in this paper, as this does not cause  additional effort in the proofs. 
\item
\label{r3}
In view of the previous point, we should also mention that for the purposes of this paper, we actually could weaken the definition of constrictedness once more -- namely, by requiring that \eqref{assssasaasasdsdsds} only holds for such metrizable compacts sets $\compacto\subseteq \mg$ that are of the form $\compacto=f([0,1])$ with $f\colon [0,1]\rightarrow \compacto$ continuous. We explicitly remark that this is equivalent to require that $\compacto=\Phi([0,1]^2)$ holds for some continuous map $\Phi\colon [0,1]^2\rightarrow \compacto$.\footnote{Observe that there exists a continuous map $\alpha\colon [0,1]\rightarrow [0,1]^2$; and that each continuous map $f\colon [0,1]\rightarrow \compacto$ admits the continuous extension $F\colon [0,1]^2\ni (x,y)\mapsto f(x)\in \compacto$.} This, in turn, is evidently equivalent to require that $\compacto=\Phi([a,b]\times [a',b'])$ holds for some continuous map $\Phi\colon [a,b]\times [a',b']\rightarrow \compacto$ with $a<b$ and $a'<b'$. 
\hspace*{\fill}$\ddagger$
\end{enumerate}
\endgroup
\end{remark}
\subsection{Statement of the Results}
\label{kjdsjdskjskjdskjsk}
In Sect.\ \ref{opfdpodfpofdpofd}, we prove the following theorem.
\begin{theorem}
\label{aoelsaoesalsaoelsa}
If $G$ is asymptotic estimate, then $G$ is $\infty$-continuous \deff $G$ is {\rm 0}-continuous (\deff 
  $G$ is locally $\mu$-convex, by Theorem \ref{LMC}).  
\end{theorem} 
We furthermore generalize Proposition 5 in \cite{RGM} (1.) by dropping both the presumption that $G$ admits an exponential map and that $\mg$ is sequentially complete, and (2.) by using the more general notion of constrictedness (cf.\ Remark \ref{dffd}.\ref{r2}).  Specifically, we show that, cf.\ Sect.\ \ref{podspodspodspo} 
\begin{proposition}
\label{fdkjfdkjdfkjfdjfdlkjfdl}
If $G$ is constricted, then $G$ is $\mathrm{k}$-confined for each $k\in \NN\sqcup\{\lip,\infty\}$.
\end{proposition}
We thus obtain from Theorem \ref{confev} that
\begin{corollary}
\label{sddssdsddssdaaaa}
Assume that $G$ is constricted and locally $\mu$-convex. Then,
\vspace{-4pt} 
\begingroup
\setlength{\leftmargini}{17pt}
\begin{enumerate}
\item
\label{confevv1}
$G$ is \emph{$C^0$-semiregular} if $G$ is \emph{sequentially complete}.
\item
\label{confevv2}
$G$ is \emph{$C^\lip$-semiregular} \deff $G$ is \emph{Mackey complete} \deff $G$ is \emph{$C^\infty$-semiregular}.
\end{enumerate}
\endgroup
\end{corollary}
Combining Theorem \ref{aoelsaoesalsaoelsa} with Corollary \ref{sddssdsddssdaaaa} and Theorem \ref{Smthn}, we obtain the following statement.
\begin{theorem}
\label{confevv111}
\noindent
\vspace{-4pt} 
\begingroup
\setlength{\leftmargini}{17pt}
\begin{enumerate}
\item
\label{confevv11}
If $G$ is constricted, then $G$ is $C^0$-regular if $G$ is locally $\mu$-convex, sequentially complete, and has integral complete Lie algebra. 
\item
\label{confevv22}
If $G$ is asymptotic estimate, then  
$G$ is $C^\infty$-regular \deff $G$ is locally $\mu$-convex, Mackey complete, and has Mackey complete Lie algebra -- In this case, $G$ is $C^k$-regular for each $k\in \NN_{\geq 1}\sqcup \{\lip,\infty\}$.
\end{enumerate}
\endgroup
\end{theorem}
\begin{proof}\let\qed\relax
\begingroup
\setlength{\leftmargini}{17pt}
\begin{enumerate}
\item
Assume that $G$ is constricted, locally $\mu$-convex, sequentially complete, and has integral complete Lie algebra. Then, 
Corollary \ref{sddssdsddssdaaaa}.\ref{confevv1} shows that $G$ is $C^0$-semiregular; and, Theorem \ref{LMC} shows that $G$ is 0-continuous. The claim thus follows from Theorem \ref{Smthn}.\ref{smthn1}.    
\item
If $G$ is asymptotic estimate and $C^\infty$-regular, then $G$ is locally $\mu$-convex by Theorem \ref{aoelsaoesalsaoelsa}, Mackey complete by Corollary \ref{sddssdsddssdaaaa}.\ref{confevv2} (confer also Theorem 2 in \cite{RGM}), and has Mackey complete Lie algebra by Theorem \ref{Smthn}.\ref{smthn2}. 

For the other direction, assume that $G$ is asymptotic estimate, locally $\mu$-convex, Mackey complete, and has Mackey complete Lie algebra. 
Then, $G$ is (in particular) constricted, hence $C^k$-semiregular for each $k\in \NN_{ \geq 1}\sqcup\{\lip,\infty\}$ by Corollary \ref{sddssdsddssdaaaa}.\ref{confevv2}. Since $G$ is locally $\mu$-convex, Theorem \ref{LMC} shows that $G$ is 0-continuous; thus, k-continuous for each $k\in \NN\sqcup\{\lip,\infty\}$. The claim is now  clear from Theorem \ref{Smthn}.\ref{smthn2}.\hspace*{\fill}$\qedsymbol$ 
\end{enumerate}
\endgroup
\end{proof}
\begin{remark}[The Lipschitz Case]
\label{sdsddssdsd}
Our convention concerning $C^\lip$-regularity is essentially due to the fact that Theorem E in \cite{HGGG} was formulated there for $k\in \NN\sqcup\{\infty\}$, but not for $k\in \NN\sqcup\{\lip,\infty\}$. More specifically,  Theorem E in \cite{HGGG} was applied in the proof of  
  Theorem 4 in \cite{RGM} (Theorem \ref{Smthn}), to deduce smoothness of $\evol_\kk$, for $k\in \NN\sqcup\{\infty\}$ from the fact that $\evol_\kk$ is of class $C^1$ (confer Corollary 13 in \cite{RGM}). 
  Now, it is to be expected that 
 the arguments in \cite{HGGG} also apply to the Lipschitz case, i.e., that Theorem E in \cite{HGGG} even holds for $k\in \NN\sqcup\{\lip,\infty\}$ (confer Remark 7 in \cite{RGM} for an alternative argument). Once this has been verified, Theorem \ref{Smthn} --  and thus, Theorem \ref{confevv111}.\ref{confevv22} -- also holds when ``of class $C^1$'' is replaced by ``smooth'' in our definition of $C^\lip$-regularity in the beginning of this section.
\hspace*{\fill}$\ddagger$
\end{remark}
\begin{remark}
\label{sdsddssdsdsasa}
We finally want to remark the following.
\begingroup
\setlength{\leftmargini}{17pt}
\begin{enumerate} 
\item
Theorem \ref{confevv111}.\ref{confevv11} is in line with the well-known fact that Banach Lie Groups are $C^0$-regular. Indeed, if $G$ is a Banach Lie group, then $\bl \cdot,\cdot\br$ is submultiplicative (thus, asymptotic estimate), $\mg$ is complete (thus, integral complete); and $G$ is locally $\mu$-convex by Proposition 14.6 in \cite{HG} (cf.\ also Appendix C.2 in \cite{RGM}), as well as sequentially complete by Example 3.2 in \cite{RGM}.
\item
Theorem \ref{confevv111}.\ref{confevv22} is in line with the theorem proven in \cite{HGIA},\footnote{Recall that this theorem was already generalized in \cite{RGM}; namely, due to Example 2.3), Example 3.3), and Point B) in Sect.\ 7.2.1 in \cite{RGM}.} stating that the unit group $G\equiv \MAU$ of a Mackey complete\footnote{Here, Mackey completeness refers to Mackey completeness of the Hausdorff locally convex vector space $(\MA,+)$.} continuous inverse algebra $\MA$, fulfilling the condition $(*)$ in \cite{HGIA}, is $C^\infty$-regular. Indeed, this condition implies that 
the Lie algebra of $\MAU$ is asymptotic estimate; and, by Example 2.3) and Example 3.3) in \cite{RGM}, it also implies that $G$ is locally $\mu$-convex and Mackey complete. Theorem \ref{confevv111}.\ref{confevv22} thus shows that $G$ is $C^\infty$-regular (even $C^\lip$-regular).
\hspace*{\fill}$\ddagger$
\end{enumerate}
\endgroup
\end{remark}

\section{Preliminaries}
\label{dsdssd}
In this section, we fix the notations, provide the basic definitions; and recall the properties of the core mathematical objects of this paper that are relevant for our discussions in Sect.\ \ref{podspodspodspo} and Sect.\ \ref{podspodspodspo1}. 
The proofs of the facts mentioned but not verified in this section can be found, e.g., in Sect.\ 3 in \cite{RGM}.

\subsection{Conventions}
\label{ksasahgsahgajdkjfdkjfd}
In this paper, Manifolds and Lie groups are always understood to be in the sense of \cite{HG} -- In particular, smooth, Hausdorff, and modeled over a Hausdorff locally convex vector space.\footnote{We explicitly refer to Definition 3.1 and Definition 3.3 in \cite{HG}. A review of the corresponding differential calculus -- including the standard differentiation rules used in this paper -- can be found, e.g., in Appendix \ref{Diffcalc} that essentially equals Sect.\ 3.3.1 in \cite{RGM}.}  
If $f\colon M\rightarrow N$ is a $C^1$-map between the manifolds $M$ and $N$, then $\dd f\colon TM \rightarrow TN$ denotes the corresponding tangent map between their tangent manifolds -- we write $\dd_xf\equiv\dd f(x,\cdot)\colon T_xM\rightarrow T_{f(x)}N$ for each $x\in M$. 
By an interval, we understand a non-empty, non-singleton connected subset $D\subseteq \RR$. 
 A curve is a continuous map $\gamma\colon D\rightarrow M$ for a manifold $M$ and an interval $D\subseteq \RR$.  
If $D\equiv I$ is open, then $\gamma$ is said to be of class $C^k$ for $k\in \NN\sqcup \{\infty\}$ \defff it is of class $C^k$ when considered as a map between the manifolds $I$ and $M$. 
If $D$ is an arbitrary interval, then $\gamma$ is said to be of class $C^k$ for $k\in \NN\sqcup \{\infty\}$ \defff $\gamma=\gamma'|_D$ holds for a $C^k$-curve $\gamma'\colon I\rightarrow M$ that is defined on an open interval $I$ containing $D$ -- we  write $\gamma\in C^k(D,M)$ in this case.  
If $\gamma\colon D\rightarrow M$ is of class $C^1$, then we denote the corresponding tangent vector at $\gamma(t)\in M$ by $\dot\gamma(t)\in T_{\gamma(t)}M$. 
These conventions also hold if $M\equiv F$ is a Hausdorff locally convex vector space, with system of continuous seminorms $\SEMM$. In this case, we let $C^\lip([r,r'],F)$ (for $r<r'$) denote the set of all Lipschitz curves; i.e., all curves $\gamma\colon [r,r']\rightarrow F$, such that
\begin{align*}
	\qq(\gamma(t)-\gamma(t'))\leq L_\qq\cdot |t-t'|\qquad\quad\forall\: t,t'\in [r,r'],\:\: \qq\in \SEMM
\end{align*}
holds for constants $\{L_\qq\}_{\qq\in \SEMM}\subseteq \RR_{\geq 0}$. For $k\in \NN\sqcup\{\lip,\infty\}$, we define
\begin{align}
\label{kjfdkjfkjfkjd}
	\textstyle\qq_\infty^\dindq(\gamma):= \sup\{\qq(\gamma^{(m)}(t))\:|\: 0\leq m\leq \dindq,\:t\in [r,r']\}\qquad\quad\forall\: \gamma\in C^k([r,r'],F)
\end{align}
for $r<r'$, $\qq\in \SEMM$, and $\dindq\llleq k$ -- which means $\dindq\leq k$ for $k\in \NN$, $\dindq\equiv 0$ for $k\equiv\lip$, and $\dindq\in \NN$ for $k\equiv \infty$ -- and let $\qqq_\infty\equiv \qqq^0_\infty$ for each $\qq\in \SEMM$.

\subsection{Lie Groups}
In this paper, $G$ will always denote an infinite dimensional Lie group  in the sense of \cite{HG} that is modeled over the Hausdorff locally convex vector space $E$. The system of continuous seminorms on $E$ will be denoted by $\SEM$; and we define
\begin{align*}
	\OB_{\mm,\varepsilon}:=\{X\in E\:|\: \mm(X)\leq\varepsilon\}\qquad\quad\forall\:\mm\in \SEM,\:\:\varepsilon>0.
\end{align*} 
We denote the Lie algebra of $G$ by $(\mg,\bl\cdot,\cdot\br)$, the identity element by $e\in G$, fix a chart 
\begin{align*}
	\chart\colon G\supseteq \U\rightarrow \V\subseteq E
\end{align*}
with $\V$ convex, $e\in \U$, and $\chart(e)=0$; 
 and identify $\mg\cong E$ via $\dd_e\chart\colon \mg\rightarrow E$.   Specifically, this means that we will write $\ppp(X)$ instead of $(\pp\cp\dd_e\chart)(X)$ for each $\pp\in \SEM$ and $X\in \mg$ in the following. 
We let $\mult\colon G\times G\rightarrow G$ denote the Lie group multiplication, $\inv\colon G\ni g\mapsto g^{-1}\in G$ the inversion, $\RT_g:=\mult(\cdot, g)$ the right translation by $g\in G$; and $\Ad\colon G\times \mg\rightarrow \mg$ the adjoint action, i.e., we have 
\begin{align*}
	\Ad(g,X)\equiv\Ad_g(X):=\dd_e\conj_g(X)\qquad\quad\text{with}\qquad\quad \conj_g\colon G\ni h\mapsto g\cdot  h\cdot g^{-1}\in G
\end{align*}
for each $g\in G$ and $X\in \mg$. We recall that for $\Ad\bl Y\br\colon G\ni g\mapsto \Ad_g(Y)\in \mg$, we have
\begin{align}
\label{podspodspods}
\dd_e\Ad\bl Y\br(X)=\bl X,Y\br\qquad\quad\forall\: X,Y\in \mg;
\end{align}
and define inductively  
\begin{align*}
\com{X}^n:=\com{X}\cp\com{X}^{n-1}\qquad\quad\forall\: n\geq 1,
\end{align*} 
with $\com{X}^0:=\id_\mg$ as well as $\com{X}\equiv \com{X}^1 \colon \mg\ni Y\mapsto \bl X,Y\br$ for each $X\in \mg$.

\subsection{The Evolution Map}
\label{kjdsjlkdsklskjd}   
In this subsection, we provide the relevant facts and definitions concerning the right logarithmic derivative and the evolution map.

\subsubsection{Basic Facts and Definitions}
The {right logarithmic derivative} is given by
\begin{align*}
	\Der\colon C^1(D,G)\rightarrow C^0(D,\mg),\qquad \mu\mapsto \dd_\mu\RT_{\mu^{-1}}(\dot \mu)
\end{align*}
for each interval $D\subseteq \RR$. It is immediate from the  definitions that the following identities hold for $\mu,\nu\in C^1(D,G)$, $g\in G$, $D'\subseteq \RR$ an interval with $D'\subseteq D$, and $\rho\colon D''\rightarrow D$ of class $C^1$ (we set $\mu^{-1}\equiv\inv\cp\mu$):
\begin{align}
\label{fgfggfsss}
\begin{split}
	\Der(\mu\cdot g)=\Der(\mu)\qquad\quad\text{and}\qquad\quad \Der(\mu|_{D'})=\Der(\mu)|_{D'}\hspace{90pt}\\[5pt]
	\Der( \mu\cp\varrho)=\dot\varrho\cdot(\Der(\mu)\cp\varrho)\hspace{170pt}\\[5pt]
	\Der(\mu\cdot \nu)= \Der(\mu)+\Ad_\mu(\Der(\nu))\qquad\text{implying}\qquad \Der(\mu^{-1}\nu)=\Ad_{\mu^{-1}}(\Der(\nu) -\Der(\mu)).\quad\:\:
\end{split}
\end{align}
We define $\DIDE_{[r,r']}:=\Der(C^1([r,r'],G))$ for $r<r'$, as well as
\begin{align*}
\textstyle\DIDE_{[r,r']}^k:=\DIDE_{[r,r']}\cap C^k([r,r'],\mg)\qquad\quad\forall\: k\in \NN\sqcup\{\lip,\infty\}.
\end{align*}
It follows from the third line in \eqref{fgfggfsss} that $\Der$ restricted to
\begin{align*}
	C^1_*([r,r'],G):=\{\mu\in C^1([r,r'],G)\:|\: \mu(0)=e\}
\end{align*}
 is injective for $r<r'$, cf.\ e.g. Lemma 9 in \cite{RGM}; hence, that the map 
$\EV\colon \bigsqcup_{\RR\ni r<r'\in \RR}\DIDE_{[r,r']}\rightarrow \bigsqcup_{\RR\ni r<r'\in \RR}C_*^1([r,r'],G)$ given by 
\begin{align*}
	\EV\colon \DIDE_{[r,r']}\rightarrow C_*^{1}([r,r'],G),\qquad\Der(\mu)\mapsto \mu\cdot \mu(r)^{-1}
\end{align*}
is defined.   
We recall that then $\EV\colon \DIDE_{[r,r']}^k\rightarrow C_*^{k+1}([r,r'],G)$ holds for $r<r'$ and $k\in \NN\sqcup\{\infty\}$, cf.\ Corollary 4 in \cite{RGM}. 

\subsubsection{The Product Integral} 
The product integral is given by
\begin{align*}
	\textstyle\innt_s^t\phi:= \EV\big(\phi|_{[s,t]}\big)(t)\in G\qquad\quad \forall \: \phi\in \bigsqcup_{ r\leq s< t\leq r'}\DIDE_{[r,r']},
\end{align*}
and we let $\innt\phi\equiv\innt_r^{r'}\phi$ as well as $\innt_c^c\phi:= e$ for $\phi\in \DIDE_{[r,r']}$ and $c\in [r,r']$. We furthermore set
\begin{align*}
	\textstyle\evol_\kk\equiv \innt\big|_{\DIDE^k_{[0,1]}} 
	\qquad\quad\forall\:  k\in \NN\sqcup\{\lip,\infty\}.
\end{align*}
Then, \eqref{fgfggfsss} implies the following elementary identities, cf.,  \cite{HGGG,MK} or Sect.\ 3.5.2 in \cite{RGM}.
\begin{custompr}{D}
\label{sddssddsxyxyxyxyyx}
Let $[r,r']$ be an interval. Then, the following assertions hold:
	\begingroup
\setlength{\leftmargini}{17pt}
{
\renewcommand{\theenumi}{\alph{enumi})} 
\renewcommand{\labelenumi}{\theenumi}
\begin{enumerate}
\item
\label{kdsasaasassaas}
For each $\phi,\psi\in \DIDE_{[r,r']}$, we have $\phi+\Ad_{\innt_r^\bullet\phi}(\psi)\in \DIDE_{[r,r']}$, with
\begin{align*}
	\textstyle\innt_r^t \phi \cdot \innt_r^t\psi=\innt_r^t (\phi+\Ad_{\innt_r^\bullet\phi}(\psi)).
\end{align*}
	\vspace{-18pt}
\item
\label{pogfpogfaaa}
For each $\phi\in \DIDE_{[r,r']}$, we have $-\Ad_{[\innt_r^\bullet\phi]^{-1}}(\phi)\in \DIDE_{[r,r']}$, with
\begin{align*}
	\textstyle\big[\innt_r^t \phi\big]^{-1}=\innt_r^t-\Ad_{[\innt_r^\bullet\phi]^{-1}}(\phi).
\end{align*}
	\vspace{-18pt}
\item
\label{pogfpogf}
\hspace{4pt}For $r=t_0<{\dots}<t_n=r'$ and $\phi\in \DIDE_{[r,r']}$, we have 
	\begin{align*}
		\textstyle\innt_r^t\phi=\innt_{t_{p}}^t\! \phi\cdot \innt_{t_{p-1}}^{t_{p}} \!\phi \cdot {\dots} \cdot \innt_{r}^{t_1}\!\phi\qquad\quad\forall\:t\in (t_p,t_{p+1}],\:\: p=0,\dots,n-1.
	\end{align*}
		\vspace{-15pt}
\item
\label{subst}
	\hspace{4pt}For $\varrho\colon [\ell,\ell']\rightarrow [r,r']$ 
of class $C^1$ and $\phi\in \DIDE_{[r,r']}$, we have $\dot\varrho\cdot (\phi\cp\varrho)\in \DIDE_{[\ell,\ell']}$, with
\begin{align*}
	 \textstyle\innt_r^{\varrho(\bullet)}\phi=\big[\innt_\ell^\bullet\dot\varrho\cdot (\phi\cp\varrho)\he\big]\cdot \big[\innt_r^{\varrho(\ell)}\phi\he\big].
\end{align*} 
\end{enumerate}}
\endgroup
\end{custompr} 
\begin{example}[The Inverse]
\label{fdpofdopdpof}
For $r<r'$ fixed, we let  
	$\varrho\colon [r,r']\ni t\mapsto r +r' -t\in [r,r']$, and define   
\begin{align}
\label{podspodspoaaa}
	\textstyle \DIDE_{[r,r']}\ni \inverse{\phi}:=\dot\varrho\cdot \phi\cp\varrho \colon [r,r']\ni t\mapsto - \phi(r+r'-t)\in \mg\qquad\quad\forall\:\phi\in \DIDE_{[r,r']}. 
\end{align}
We let $[\ell,\ell']\equiv[r,r']$; and obtain from Part \ref{subst} of Proposition \ref{sddssddsxyxyxyxyyx} that
\begin{align}
\label{pfifpfpofdpofd}
	\textstyle e=\innt_r^{\varrho(r')}\phi\stackrel{\ref{subst}}{=}\big[\innt_{r}^{r'} \inverse{\phi}\big] \cdot \big[\innt_r^{r'}\phi  \big]\qquad\quad\text{holds, thus}\qquad\quad [\innt \phi]^{-1}=\innt \inverse{\phi},
\end{align}
which will be useful for our argumentation in Sect.\ \ref{pogfpogfa}.
\hspace*{\fill}$\ddagger$
\end{example}
\noindent 
Now, for $r<r'$ and $k\in \NN\sqcup\{\infty\}$ fixed, we let $\DP^k([r,r'],\mg)$ denote the set of all $\phi\colon [r,r']\rightarrow \mg$, such that there exist $r=t_0<{\dots}<t_m=r'$ ($m\geq 1$) and $\phi[p]\in \DIDE^k_{[t_p,t_{p+1}]}$ for $p=0,\dots,m-1$ with
\begin{align}
\label{opopooppo}
	\phi|_{(t_p,t_{p+1})}=\phi[p]|_{(t_p,t_{p+1})}\qquad\quad\forall\: p=0,\dots,m-1.
\end{align}  
In this situation, we define $\innt_r^r\phi:=e$, as well as 
\begin{align*}
	\textstyle\innt_r^t\phi&\textstyle:=\innt_{t_{p}}^t \phi[p] \cdot \innt_{t_{p-1}}^{t_p} \phi[p-1]\cdot {\dots} \cdot \innt_{r}^{t_1}\phi[0]\qquad\quad \forall\: t\in (t_{p}, t_{p+1}],\:\: p=0,\dots,m-1.
\end{align*} 
A standard refinement argument in combination with Part \ref{pogfpogf} of Proposition \ref{sddssddsxyxyxyxyyx} then shows that this is well defined, i.e., independent of any choices we have made. 
\subsubsection{Continuity and Semiregularity}
For $k\in \NN\sqcup\{\lip,\infty\}$, we say that $G$ is  
\begingroup
\setlength{\leftmargini}{12pt}
\begin{itemize}
\item
k-continuous \defff $\evol_\kk$ 
is \emph{$C^k$-continuous}; thus, continuous  
w.r.t.\ the seminorms \eqref{kjfdkjfkjfkjd}.
\vspace{2pt} 
\item
\emph{$C^k$-semiregular} \defff $\DIDE^k_{[0,1]}=C^k([0,1],\mg)$ holds.
\item
\emph{$C^k$-regular} \defff $G$ is $C^k$-semiregular, and $\evol_\kk$ 
is smooth w.r.t.\ the $C^k$-topology. 
\end{itemize}
\endgroup
\noindent
We define $\phi_X\colon \RR\ni t\mapsto X\in \mg$ for each $X\in \mg$; and remark that
\begin{remark} 
It is straightforward from the properties of the right logarithmic derivative (cf. Lemma 11 in \cite{RGM}) that for each $X\in \mg$ with $\phi_X|_{[0,1]}\in \DIDE_{[0,1]}$, we have $\phi_X|_{[r,r']}\in \DIDE_{[r,r']}$ for all $r<r'$. 
\hspace*{\fill}$\ddagger$
\end{remark}
\noindent
We say that $G$ admits an exponential map \defff $\phi_X|_{[0,1]}\in \DIDE_{[0,1]}$ holds for each $X\in \mg$; and define $\exp_G\colon \mg\ni X\mapsto \innt \phi_X|_{[0,1]}$ in this case.\footnote{Observe that we do not impose any differentiability\slash smoothness presumptions on $\exp_G$.} 

\subsection{The Riemann Integral}
\label{opsdpods}
Let $F$ be a Hausdorff locally convex vector space with system of continuous seminorms $\SEMM$, and completion $\comp{F}$. 
The Riemann integral of $\gamma\in C^0([r,r'],F)$ (for $r<r'$) is  denoted by 
$\int \gamma(s) \:\dd s\in \comp{F}$.\footnote{We explicitly remark at this point that the Riemann integral can be defined exactly as in the finite dimensional case; namely, as a limit of Riemann sums. Details can be found, e.g., in Sect.\ 2 in  \cite{COS}.} 
We define 
\begin{align*}
	\textstyle\int_a^b \gamma(s)\:\dd s:= \int \gamma|_{[a,b]}(s) \:\dd s\qquad\quad\:\int_b^a \gamma(s) \:\dd s:= - \int_a^b \gamma(s) \:\dd s\qquad\quad\:
	 \int_c^c \gamma(s)\: \dd s:=0\qquad
\end{align*}
for $r\leq a<b\leq r'$, $c\in [r,r']$. Clearly, the Riemann integral is linear, with	
\begin{align*}
	\textstyle\int_a^c \gamma(s) \:\dd s&\textstyle=\int_a^b \gamma(s)\:\dd s+ \int_b^c \gamma(s)\:\dd s\qquad\quad \forall\: r\leq a< b< c\leq r'.
\end{align*}
Moreover, we have
	 \begin{align}
	\label{isdsdoisdiosd}
		\textstyle\gamma(t)-\gamma(r)\hspace{4pt}&\textstyle=\int_r^t \dot\gamma(s)\:\dd s\\
	\label{isdsdoisdiosd1}
		\qq(\gamma(t)-\gamma(r))&\textstyle\leq \int_r^t \qq(\dot\gamma(s))\: \dd s\qquad\quad\forall\: \qq\in \SEMM,
	\end{align}
	for all $\gamma\in C^1([r,r'],F)$ and $t\in [r,r']$; as well as
	\begin{align}
\label{substitRI}
	\textstyle\int_r^{\varrho(\bullet)} \gamma(s)\: \dd s=\int_\ell^\bullet \dot\varrho(s)\cdot \gamma(\varrho(s))\:\dd s
\end{align}
for each $\gamma\in C^0([r,r'],F)$, and each $\varrho\colon [\ell,\ell'] \rightarrow [r,r']$ of class $C^1$ with $\varrho(\ell)=r$ and $\varrho(\ell')=r'$.
	
\subsection{Some Estimates}
We recall that, cf.\ Sect.\ 3.4.1 in \cite{RGM} and Corollary 1 in Sect.\ 3.2 in \cite{RGM}:
\begingroup
\setlength{\leftmargini}{19pt}
{
\renewcommand{\theenumi}{\roman{enumi})} 
\renewcommand{\labelenumi}{\theenumi}
\begin{enumerate}
\item
\label{as1}
For each compact $\compact\subseteq G$, and each $\vv\in \SEM$, there exists some $\vv\leq \ww\in \SEM$ with
\begin{align*}
	\vvv\cp \Ad_g\leq \www\qquad\quad\forall\: g\in \compact.
\end{align*} 
\item
\label{as2}
Assume that $\im[\mu]\subseteq \U$ holds for $\mu\in C^1([r,r'],G)$. Then, we have 
\begin{align}
\label{kldlkdldsl}
	\Der(\mu)=\dermapdiff(\chart\cp\mu,\partial_t(\chart\cp\mu))
\end{align}      
for the smooth map
\begin{align*}
	\dermapdiff\colon& \V\times E\rightarrow \mg,\qquad (x,X)\mapsto \dd_{\chartinv(x)}\RT_{[\chartinv(x)]^{-1}}(\dd_x\chartinv(X)).
\end{align*}
Moreover, for each $n\in \NN$, the map\footnote{Here, $[\partial_1]^{n}$ denotes the $n$-times iterated partial derivative w.r.t.\ the first argument, confer also Part \ref{productrule} \eqref{fdfdddfdfdfdddhjfdfdff} in Appendix \ref{Diffcalc}.} 
\begin{align*}
	\dermapdiff[n]:=[\partial_1]^n\dermapdiff\colon \V\times E^{n+1}\rightarrow \mg
\end{align*}
is continuous, as well as $n+1$-multilinear in the second factor. Then, for $\ww\in \SEM$ and $\dindq\in \NN$ fixed, there exists some $\ww\leq \qq\in \SEM$, such that
\begin{align}
\label{omegaklla}
	(\www\cp\dermapdiff[n])(x,X_1,\dots,X_{n+1})\leq \qq(X_1)\cdot {\dots}\cdot \qq(X_{n+1})
\end{align} 
holds for all $x\in \OB_{\qq,1}$, $X_1,{\dots},X_{n+1}\in E$, and $0\leq n\leq \dindq$.
\item
\label{as3}
Assume that $\im[\mu]\subseteq \U$ holds for $\mu\in C^1([r,r'],G)$. Then, we have
\begin{align}
\label{ixxxsdsdoisdiosd}
	\partial_t\he(\chart\cp\mu)&=\dermapinvdiff(\chart\cp\mu,\Der(\mu))\\
\label{ixxxsdsdoisdiosda}
	\stackrel{\eqref{isdsdoisdiosd}}{\Longrightarrow}\qquad\qquad \chart\cp\mu&\textstyle =\int_r^\bullet \dermapinvdiff((\chart\cp\mu)(s),\Der(\mu)(s))\: \dd s\qquad\qquad
\end{align}   
for the smooth map
\begin{align*}
	\dermapinvdiff\colon& \V\times \mg\rightarrow E,\qquad (x,X)\mapsto \big(\dd_{\chartinv(x)}\chart\cp \dd_{e}\RT_{\chartinv(x)}\big)(X)
\end{align*}
that is linear in the second argument. 
Then, for each $\qq\in \SEM$, there exists some $\qq\leq\mm\in \SEM$ with
\begin{align}
\label{sadsndsdsnmdsds}
	(\qq\cp \dermapinvdiff)(x,X)\leq \mm(X)\qquad\quad\forall\: x\in \OB_{\mm,1},\:\: X\in \mg.
\end{align}
For each $\mu\in C^1([r,r'],G)$ with $\im[\chart\cp\mu]\subseteq \OB_{\mm,1}$, we thus obtain from \eqref{ixxxsdsdoisdiosd}, \eqref{isdsdoisdiosd}, and \eqref{isdsdoisdiosd1} that
\begin{align}
\label{sadsndsdsnmdsdsa}
\begin{split}
	\textstyle\qq(\chart\cp\mu)&\stackrel{}{=}\textstyle\qq\big(\int_r^\bullet \dermapinvdiff((\chart\cp\mu)(s),\Der(\mu)(s))\:\dd s\big)\\
	&\textstyle\leq \int \mmm(\Der(\mu)(s))\:\dd s\\
	&\textstyle\leq |r'-r|\cdot \mmm_\infty(\Der(\mu)).
\end{split}
\end{align}
\end{enumerate}}
\endgroup

\subsection{Completeness and Approximation}
\label{uepdspodspodsapoa}
We now finally list some definitions from \cite{RGM} that will occur in Sect.\ \ref{podspodspodspo}. We furthermore supplement Lemma 29 in \cite{RGM} by a technical detail (that is already part of the proof given there)   
in order to make it compatible with the definition of constrictedness \eqref{assssasaasasdsdsds} given in Sect.\ \ref{lkfdlkfdlkdflkfdlfd}, cf.\ also Remark \ref{dffd}.\ref{r2}.
\subsubsection{Completeness}
\label{nmsnmdsnmnmds}
Let $\{g_n\}_{n\in \NN}\subseteq G$ be a sequence.
\begingroup
\setlength{\leftmargini}{11pt}
\begin{itemize}
\item
$\{g_n\}_{n\in \NN}$ is said to be a {Cauchy sequence} \defff for each $\pp\in \SEM$ and $\varepsilon>0$, there exists some $p\in \NN$ with
	$(\pp\cp\chart)(g_m^{-1}\cdot g_n)\leq \varepsilon$ for all $m,n\geq p$.
\item
$\{g_n\}_{n\in \NN}$ is said to be a {Mackey-Cauchy sequence} \defff 
 \begin{align}
 \label{kjdskjskjds}
 	(\pp\cp\chart)(g^{-1}_m\cdot g_{n})\leq \mackeyconst_\pp\cdot \lambda_{m,n}\qquad\quad\forall\: m,n\geq \mackeyindex_\pp,\:\: \pp\in\SEM
 \end{align}
 holds for certain $\{\mackeyconst_\pp\}_{\pp\in \SEM}\subseteq \RR_{\geq 0}$, $\{\mackeyindex_\pp\}_{\pp\in \SEM}\subseteq \NN$, and $\RR_{\geq 0}\supseteq \{\lambda_{m,n}\}_{(m,n)\in \NN\times \NN}\rightarrow 0$.  
\end{itemize}
\endgroup
\noindent
Both definitions are independent of the explicit choice of the chart $\chart$, because coordinate changes are locally Lipschitz continuous (for details confer Remark 3 and Appendix D.1 in \cite{RGM}).
\vspace{6pt}

\noindent
Then,  
\begingroup
\setlength{\leftmargini}{11pt}
\begin{itemize}
\item
$G$ is said to be {sequentially complete} \defff each Cauchy sequence in $G$ converges in $G$.
\item
$G$ is said to be {Mackey complete} \defff each Mackey-Cauchy sequence in $G$ converges in $G$.
\end{itemize}
\endgroup
\begin{remark}
\label{nmdfnmfdnmfdfd} 
Evidently, each Mackey-Cauchy sequence is a Cauchy sequence; so that $G$ is Mackey complete if $G$ is sequentially complete -- Of course, the converse statements usually do not hold. 
 We furthermore remark that \eqref{kjdskjskjds} is equivalent to require that 
 \begin{align*}
 	(\pp\cp\chart)(g^{-1}_m\cdot g_{n})\leq \mackeyconst_\pp\cdot \lambda_{m,n}\qquad\quad\forall\: m,n\in \NN,\:\: \pp\in\SEM
 \end{align*}
 holds for certain $\{\mackeyconst_\pp\}_{\pp\in \SEM}\subseteq \RR_{\geq 0}$, and $\RR_{> 0}\supseteq \{\lambda_{m,n}\}_{(m,n)\in \NN\times \NN}\rightarrow 0$ -- Hence, the constants $\{\mackeyindex_\pp\}_{\pp\in \SEM}$ can be circumvented if $\lambda_{m,n}>0$ is presumed for all $m,n\in \NN$. Evidently, then the latter definition is equivalent to the definition traditionally used in the  literature (cf., e.g., Sect.\ 2 in \cite{COS}).  
 Anyhow, the formulation \eqref{kjdskjskjds} has been introduced in \cite{RGM} for practical reasons; and we here will stick to these conventions for reasons of consistency.  
 \hspace*{\fill}$\ddagger$
\end{remark}
\begingroup
\setlength{\leftmargini}{11pt}
\begin{itemize}
\item
We say that $\mg$ is {sequentially\slash Mackey complete} \defff $\mg$ is {sequentially\slash Mackey complete} when considered as the Lie group $(\mg,+)$.\footnote{According to Theorem  2.14 in \cite{COS}, in the Mackey case this definition is equivalent to the definition of Mackey completeness used in Sect.\ \ref{sdlkdslkdslkds}.} 
\item
We say that $\mg$ is {integral complete} \defff the Riemann integral $\int \phi(s)\: \dd s\in \mg$ exists for each $\phi\in C^0([0,1],\mg)$.
\end{itemize}
\endgroup

\subsubsection{Approximation}
For $r<r'$, we  
let $\CP^0([r,r'],\mg)$ denote the set of all $\gamma\colon [r,r']\rightarrow \mg$ such that there exist $r=t_0<{\dots}<t_n=r'$ as well as $\gamma[p]\in C^0([t_p,t_{p+1}],\mg)$ for $p=0,\dots,n-1$ with
\begin{align*}
	\gamma|_{(t_p,t_{p+1})}=\gamma[p]|_{(t_p,t_{p+1})}\qquad\quad\forall\: p=0,\dots,n-1.
\end{align*} 
In analogy to Sect.\ \ref{nmsnmdsnmnmds}, we say that  $\{\phi_n\}_{n\in \NN}\subseteq \CP^0([r,r'],\mg)$ is a
\begingroup
\setlength{\leftmargini}{11pt}
\begin{itemize}
\item
{Cauchy sequence} \defff for each $\pp\in \SEM$ and $\varepsilon>0$, there exists some $p\in \NN$ with
	$\ppp_\infty(\phi_m-\phi_n)\leq \varepsilon$ for all $m,n\geq p$.
\item
{Mackey-Cauchy sequence} \defff 
\begin{align*}
	\ppp_\infty(\phi_m-\phi_n)\leq \mackeyconst_\pp\cdot \lambda_{m,n} \qquad\quad\forall\: m,n\geq \mackeyindex_\pp,\:\:\pp\in \SEM
\end{align*}
holds for certain $\{\mackeyconst_\pp\}_{\pp\in \SEM}\subseteq \RR_{\geq 0}$, $\{\mackeyindex_\pp\}_{\pp\in \SEM}\subseteq \NN$, and $\RR_{\geq 0}\supseteq \{\lambda_{m,n}\}_{(m,n)\in \NN\times \NN}\rightarrow 0$.
\end{itemize}
\endgroup
\noindent
We say that $\CP^0([r,r'],\mg)\supseteq \{\phi_n\}_{n\in \NN}\rightarrow \phi\in \CP^0([r,r'],\mg)$ {converges uniformly} \defff 
\begin{align*}
	\textstyle\lim_{n\rightarrow \infty}\ppp_\infty(\phi-\phi_n)=0 \qquad\quad\text{holds for each}\qquad\quad \pp\in \SEM.
\end{align*}
Evidently, we have $\DP^0([r,r'],\mg),C^0([r,r'],\mg)\subseteq \CP^0([r,r'],\mg)$.
\begin{lemma}
\label{jshjsdahjsd}
\begingroup
\setlength{\leftmargini}{17pt}
\begin{enumerate}
\item
For each $\phi\in C^0([r,r'],\mg)$, there exists a Cauchy sequence $\{\phi_n\}_{n\in \NN}\subseteq \DP^\infty([r,r'],\mg)$ with $\{\phi_n\}_{n\in \NN}\rightarrow \phi$ uniformly, such that $\bigcup_{n\in \NN}\im[\phi_n]$ is contained in a metrizable compact subset $\compacto\subseteq \mg$.
\item
For each $\phi\in C^\lip([r,r'],\mg)$, there exists a Mackey-Cauchy sequence $\{\phi_n\}_{n\in \NN}\subseteq \DP^\infty([r,r'],\mg)$  with $\{\phi_n\}_{n\in \NN}\rightarrow \phi$ uniformly, such that $\bigcup_{n\in \NN}\im[\phi_n]$ is contained in a metrizable compact subset $\compacto\subseteq \mg$.
\end{enumerate}
\endgroup
\end{lemma}
\begin{proof}
In the proof of Lemma 29 in \cite{RGM}, the approximating sequence $\{\phi_n\}_{n\in \NN}\subseteq \DP^\infty([r,r'],\mg)$ is  constructed from the continuous map
\begin{align*}
	\Phi(t,t'):= \w(t\cdot \dd_e\chart(\phi(t')),\dd_e\chart(\phi(t')))\qquad\quad\forall\: (t,t')\in [0,\Delta]\times [r,r'],
\end{align*}
for some suitably small $\Delta>0$. This is done in such a way that $\im[\phi_n]\subseteq \im[\Phi]=:\compacto$ holds for each $n\in \NN$.\footnote{Roughly speaking, one chooses suitable decompositions  $r=t_{n,0}<{\dots}<t_{n,n}=r'$; and defines $\phi_n|_{[t_{n,p},t_{n,p+1})}:=\Phi(\cdot - t_{n,p},t_{n,p})$ for $p=0,\dots,n-1$, as well as  $\phi|_{[t_{n,n-1},t_{n,n}]}:=\Phi(\cdot-t_{n,n-1},t_{n,n-1})$.} The claim is thus clear from Remark \ref{dffd}.\ref{r3}. 
\end{proof}
\subsection{Piecewise Differentiable Curves} 
By an element in $\CP_0^1([r,r'],\mg)$ (for $r<r'$), we understand a decomposition $r=t_0<{\dots}< t_q=r'$ ($q\geq 1$) together with a collection $\alpha\equiv\{\alpha[p]\}_{0\leq p\leq q-1}$ of maps 
$\alpha[p]\in C^1([t_p,t_{p+1}],\mg)$ for $p=0,\dots,q-1$, such that
\begin{align}
\label{lkfdlkfdlkfd}
	\alpha[p](t_{p+1})=\alpha[p+1](t_{p+1})\qquad\quad\forall\: p=0,\dots,q-2
\end{align}
holds. In this case, we  
define $\ul{\alpha}\in C^0([r,r'],\mg)$ by   
\begin{align*}
	\ul{\alpha}|_{[t_p,t_{p+1})}&:=\alpha[p]|_{[t_p,t_{p+1})}\qquad\forall\: p=0,\dots,q-2\qquad\quad\text{as well as}\qquad\quad \ul{\alpha}|_{[t_{q-1},t_{q}]}:=\alpha[q-1];
\end{align*} 
and, for each $\phi\in \DIDE_{[r,r']}$, we let 
\begin{align}
\textstyle\mm_\infty(\phi,\alpha)&\textstyle:= \max\{0\leq p\leq q-1\:|\: \mm_\infty(\dot\alpha[p]-\bl\phi|_{[t_p,t_{p+1}]},\alpha[p]\br)\}\qquad\text{for each}\qquad \mm \in \SEM
\nonumber\\[5pt]
\label{ykjdsjkjds}
\begin{split}
\textstyle\AI(\phi,\alpha)(t)&\textstyle:= \int_{t_p}^t \Ad_{[\innt_r^s\phi]^{-1}}(\dot\alpha[p](s)-\bl\phi(s),\alpha[p](s)\br) \:\dd s\\
	 &\textstyle\quad\hspace{4pt}+ \int_{t_{p-1}}^{t_p} \Ad_{[\innt_{r}^{s}\phi]^{-1}}(\dot\alpha[p-1](s)-\bl\phi(s),\alpha[p-1](s)\br) \:\dd s\\
	 &\textstyle\quad\hspace{4pt}+ {\dots}\\
	  &\textstyle\quad\hspace{4pt}+ \int_r^{t_1} \Ad_{[\innt_r^s\phi]^{-1}}(\dot\alpha[0](s)-\bl\phi(s),\alpha[0](s)\br) \:\dd s
\end{split}
\end{align}
with
$$
t \in 
\begin{cases}
[t_p,t_{p+1})\qquad \text{for}\qquad 0 \leq p\leq q-2,\\
[t_{q-1},t_{q}]\qquad\hspace{1.5pt} \text{for}\qquad p\equiv q-1.
\end{cases}
$$
Moreover, we write $\{\alpha_n\}_{n\in \NN} \slim{\phi} 0$ for $\{\alpha_n\}_{n\in \NN}\subseteq \CP_0^1([r,r'],\mg)$ \defff 
\begin{align*}
	\textstyle\lim_n \mm_\infty(\phi,\alpha_n)=0\qquad\:\:\text{holds, for each}\qquad\:\: \mm\in \SEM;
\end{align*}
and we will consider $C^1([r,r'],\mg)$ as a subset of $\CP_0^1([r,r'],\mg)$ in the obvious way.

\section{The Adjoint Action}
\label{sdddsdsdsds}
In this section, we prove an approximation property of the adjoint action; from which we then derive 
certain estimates under additional continuity presumptions imposed on the Lie bracket.    
\subsection{Basic Facts and Definitions}
For $\phi\in \DP^0([r,r'],\mg)$ fixed, we define 
\begin{align*}
	\Add_{\phi^\pm}\colon [r,r']\ni t\mapsto \Add^t_{\phi^\pm}:= \Ad_{[\innt_r^t \phi]^\pm};
\end{align*}
and let $\Add_\phi\equiv \Add_{\phi^+}$  
 as well as $\Add^t_\phi\equiv \Add^t_{\phi^+}$ for each $t\in [r,r']$.  
Moreover, for 
$\phi_1,\dots,\phi_d\in \DIDE_{[r,r']}$, $\psi_1,\dots,\psi_{d+2}\in C^0([r,r'],\mg)$, and $m_1,\dots,m_{d+1}\in \NN$ given, we consider 
\begin{align}
\label{ndsnsnmsdnmdsnm}
	\beta=\big(\com{\psi_1}^{m_1}\cp \Add_{\phi_1}\cp{\dots}\cp \com{\psi_d}^{m_d}\cp \Add_{\phi_d}\cp \com{\psi_{d+1}}^{m_{d+1}}\big)(\psi_{d+2})
\end{align}
as an element in $C^0([r,r'],\mg)$. We observe that $\beta\in C^{k+1}([r,r'],\mg)$ holds for $k\in \NN\sqcup \{\infty\}$ if we have $\phi_1,\dots,\phi_d\in \DIDE^k_{[r,r']}$ and $\psi_1,\dots,\psi_{d+2}\in C^{k+1}([r,r'],\mg)$. 
\vspace{6pt}

\noindent
We let $\frac{\dd}{\dd h}\big|^{+}_{h=0}$ denote the right sided  derivative, and conclude from \eqref{podspodspods} that
\begin{align}
\label{nmsndsnmds}
	\textstyle\frac{\dd}{\dd h}\big|^{+}_{h=0}\: \Ad_{\mu(\ell+h)}(Y)=\bl\dot\mu(\ell),Y\br\qquad\quad\forall\: Y\in \mg
\end{align}
holds for each $\mu \in C^1([\ell,\ell'])$ with $\mu(\ell)=e$. 
We obtain the following statements.
\begin{lemma}
\label{khdhjdhdkjkjdkjdjkd}
For $r<r'$ fixed, we have
\begin{align}
\label{vcvcvcvcvc}
	\partial_t (\Add_\phi(\psi))=\bl \phi,\Add_\phi(\psi)\br + \Add_\phi(\dot\psi)\qquad\quad\forall\: \phi\in \DIDE_{[r,r']},\:\: \psi\in C^1([r,r'],\mg).
\end{align}
\end{lemma}
\begin{proof}
For $t\in (r,r')$ fixed, we choose $\Delta>0$ with $t+[0,1]\cdot \Delta\subseteq (r,r')$; and obtain from \eqref{nmsndsnmds}, Part \ref{pogfpogf} of Proposition \ref{sddssddsxyxyxyxyyx}, and the parts  \ref{linear}, \ref{chainrule}, \ref{productrule} of Proposition \ref{iuiuiuiuuzuzuztztttrtrtr} that
\begin{align*}
	\textstyle \frac{\dd}{\dd h}\big|_{h=0}\:(\Add_\phi(\psi))(t+h)&\textstyle=\frac{\dd}{\dd h}\big|^{+}_{h=0}\: \Ad_{\innt_r^{t+h}\phi}(\psi(t+h))\\
	&=\textstyle \frac{\dd}{\dd h}\big|^{+}_{h=0}\: \Ad_{\innt_t^{t+h}\phi}\big(\Add_\phi^t(\psi(t+h))\big)\\
	&=\textstyle  \bl\phi(t),\Add^t_\phi(\psi(t))\br+ \Add^t_\phi(\dot\psi(t))
\end{align*}
holds. This shows that \eqref{vcvcvcvcvc} holds on $(r,r')$, so that the claim follows from continuity of both sides of  \eqref{vcvcvcvcvc}.
\end{proof}
\begin{lemma}
\label{fdfddfdf}
Let $\alpha\equiv\{\alpha[p]\}_{0\leq p\leq q-1}\in \CP_0^1([r,r'],\mg)$ and $\phi\in \DIDE_{[r,r']}$ be given. Then, (recall  \eqref{ykjdsjkjds})
\begin{align*}
	\textstyle\beta-\beta(r)=\AI(\phi,\alpha)\qquad\quad\text{holds for}\qquad\quad \beta:=\Add_{\phi^{-1}}\cp\ul{\alpha}\in C^0([r,r'],\mg).
\end{align*} 
\end{lemma}
\begin{proof}
Let 
$r=t_0<{\dots}<t_q=r'$ be the decomposition that corresponds to $\alpha\equiv\{\alpha[p]\}_{0\leq p\leq q-1}$. Then, for  
$t\in (t_p,t_{p+1})$, we choose $\Delta>0$ with $t+[0,1]\cdot \Delta\subseteq (t_p,t_{p+1})$; and 
		obtain from Part \ref{pogfpogfaaa} of Proposition \ref{sddssddsxyxyxyxyyx} that
\begin{align*}
	\textstyle\big[\innt_t^{t+h}\phi\:\big]^{-1}=\innt_t^{t+h}-\Ad_{[\innt_t^\bullet \phi]^{-1}}(\phi)\qquad\quad\forall\: 0<h\leq\Delta
\end{align*}
holds. 
Then, by \eqref{nmsndsnmds}, Part \ref{pogfpogf} of Proposition \ref{sddssddsxyxyxyxyyx}, and the parts \ref{linear}, \ref{chainrule}, \ref{productrule} of Proposition \ref{iuiuiuiuuzuzuztztttrtrtr}, we have
\begin{align*}
	\textstyle\dot\beta(t)&\textstyle=\frac{\dd}{\dd h}\big|^{+}_{h=0}\:\Ad_{[\innt_r^{t+h}\phi]^{-1}}(\alpha[p](t+h))\\
	&\textstyle=\frac{\dd}{\dd h}\big|^{+}_{h=0}\: \Add^t_{\phi^{-1}}\big(\Ad_{[\innt_t^{t+h}\phi]^{-1}}(\alpha[p](t+h))\big)\\
	&=\Add^t_{\phi^{-1}}(\bl-\phi(t),\alpha[p](t)\br+ \dot\alpha[p](t))\\
	&\textstyle= \Add^t_{\phi^{-1}}(\dot\alpha[p](t)-\bl\phi(t),\alpha[p](t)\br).
\end{align*}
For each $[\tau,t]\subseteq (t_p,t_{p+1})$, we thus obtain from \eqref{isdsdoisdiosd} that 
\begin{align}
\label{podsopdspodspodspo}
	\textstyle\beta(t)-\beta(\tau)= \int_\tau^t \Add^s_{\phi^{-1}}(\dot\alpha[p](s)-\bl\phi(s),\alpha[p](s)\br) \:\dd s
\end{align}
holds. The claim now follows from \eqref{ykjdsjkjds}, because $\beta$ is continuous, and because the integrand on the right hand side of  \eqref{podsopdspodspodspo} is continuous on $[t_p,t_{p+1}]$ for $p=0,\dots,q-1$.
\end{proof}
\begin{remark}[The Adjoint Equation]
The adjoint equation reads 
\begin{align*}
	\dot\alpha=\bl\phi,\alpha\br\quad\text{with}\quad \alpha(r)=Y\qquad\quad\text{for}\qquad\quad \alpha\in C^1([r,r'],\mg),\:\: \phi\in C^0([r,r'],\mg),\:\: Y\in \mg.
\end{align*}
Due to \eqref{vcvcvcvcvc} (for $\psi\equiv \phi_Y|_{[r,r']}$ there), for $\phi\in \DIDE_{[r,r']}$ this is solved by $\alpha=\Add_\phi(Y)$.  Lemma \ref{fdfddfdf} then already implies that this solution is unique. 
\vspace{6pt}

\noindent 
	In fact, let $\alpha\in C^1([r,r'],\mg)$ be given, with $\dot\alpha=\bl\phi,\alpha\br$ (hence, $\dot\alpha-\bl\phi,\alpha\br=0$) and $\alpha(r)=Y$. We define $\beta:=\Add_{\phi^{-}}\cp\alpha$, and obtain from Lemma \ref{fdfddfdf} that $0=\beta-\beta(r)=\beta -Y$ holds. This yields
\begin{align*}
	\Add^t_\phi(Y)=\Add_\phi^t(\beta(t))=\Add_{\phi}^t(\Add_{\phi^-}^t(\alpha(t)))=\alpha(t)\qquad\quad\forall\: t\in [r,r'],
\end{align*}
which proves uniqueness.
\hspace*{\fill}$\ddagger$
\end{remark}
\subsection{Uniform Approximation}
We now will use Lemma \ref{fdfddfdf} to prove certain approximation properties of maps of the form \eqref{ndsnsnmsdnmdsnm}. We start with the following observation. 
\begin{lemma}
\label{podfopdfdofpdfpofdpo}
Assume we are given $\phi\in \DIDE_{[r,r']}$ and $\{\alpha_n\}_{n\in \NN}\subseteq \CP_0^1([r,r'],\mg)$ with $\{\ul{\alpha_n}(r)\}_{n\in \NN}\rightarrow Y\in \mg$ and $\{\alpha_n\}_{n\in \NN}\slim{\phi} 0$. 
Then,  $\{\ul{\alpha_n}\}_{n\in \NN}\rightarrow \Add_{\phi}(Y)$ converges uniformly.
\end{lemma}
\begin{proof}
	For $\pp\in \SEM$ fixed, we choose seminorms $\pp\leq \qq\leq \mm$ with, cf.\ \ref{as1} 
	\begin{align}
	\label{opspodspodspodsposdpd}
		\ppp\cp \Add_{\phi}\leq \qqq\qquad\quad\:\:\text{and}\qquad\quad\:\:\qqq\cp \Add_{\phi^{-1}}\leq \mmm.
	\end{align}		
	Moreover, we let $\beta_n:=\Add_{\phi^{-1}}\cp\ul{\alpha_n}$ for each $n\in \NN$, fix $\varepsilon>0$, and choose 
	 $n_\varepsilon\in \NN$ with\footnote{Observe that $\beta_n(r)=\ul{\alpha_n}(r)$ holds for each $n\in \NN$.} 
	\begin{align}
		 	\label{asjasllklksaklasasasa}
	|r'-r|\cdot \mmm_\infty(\phi,\alpha_n)+ 	\qqq(\beta_n(r)-Y) < \varepsilon\qquad\quad\forall\: n\geq n_\varepsilon.
	\end{align}		 
	 Then, by Lemma \ref{fdfddfdf}, we have
	 \begin{align}
	 \label{asjasllklksaklas}
	 \begin{split}
	 	\textstyle\qqq_\infty(\beta_n-\beta_n(r))&\textstyle= \qqq_\infty(\AI(\phi,\alpha_n))\stackrel{\eqref{isdsdoisdiosd1}, \eqref{opspodspodspodsposdpd}}{\leq} |r'-r|\cdot  \mmm_\infty(\phi, \alpha_n).
	 	\end{split}
	 \end{align}
	 We thus obtain for $n\geq n_\varepsilon$ that
	 \vspace{-5pt} 
	\begin{align*}
	\ppp_\infty(\ul{\alpha_n}-\Add_{\phi}(Y))=\ppp_\infty(\Add_{\phi}(\beta_n-Y))&\stackrel{\eqref{opspodspodspodsposdpd}}{\leq} \qqq_\infty(\beta_n-Y)\\[8pt]
	&\hspace{3pt}	\leq \hspace{3pt}\qqq_\infty(\beta_n-\beta_n(r)) +\qqq(\beta_n(r)-Y)\\
	&\stackrel{\eqref{asjasllklksaklas}}{\leq} |r'-r|\cdot \mmm_\infty(\phi,\alpha_n)+\qqq(\beta_n(r)-Y) \\
	& \stackrel{\eqref{asjasllklksaklasasasa}}{<} \varepsilon
\end{align*}
	 \vspace{-5pt} 
holds, which proves the claim.
\end{proof}
\begin{corollary}
\label{fiofiofdoifdio}
Assume we are given $X,Y\in \mg$ with 
$\phi_{X}|_{[0,1]}\in \DIDE_{[0,1]}$, such that to each $\vv\in \SEM$, there exists some $C\geq 0$ with $\vvv(\com{X}^n(Y))\leq C^n$ for each $n\geq 1$. 
Then, we have
\begin{align*}
	\textstyle [r,r']\ni t\mapsto \Ad_{\innt_r^t\phi_X}(Y)
	=\sum_{k=0}^\infty \frac{(t-r)^k}{k!}\cdot \com{X}^k(Y)
\end{align*}
for each $r<r'$, whereby the right hand side converges uniformly. 
\end{corollary} 
\begin{proof}
	For $C^1([r,r'],\mg)\ni\alpha_n\colon [r,r']\ni t\mapsto \sum_{k=0}^n \frac{(t-r)^k}{k!}\cdot \com{X}^k(Y)$, we have $\ul{\alpha_n}(r)=Y$ as well as
	\begin{align*}
		\textstyle\vvv(\dot\alpha_n-\bl\phi_X|_{[r,r']},\alpha_n\br)= \vvv\big(-\frac{(\cdot-r)^n}{n!}\cdot \com{X}^{n+1}(Y)\big)\leq C\cdot \frac{(|r'-r|\cdot C)^n}{n!}\qquad\quad\forall\: \vv\in \SEM,\:\: n\in \NN.
	\end{align*}
	The claim thus follows from Lemma \ref{podfopdfdofpdfpofdpo}.
\end{proof}
\begin{remark}[Duhamel's formula]
\label{pofpopofdoppofd}
Assume that $G$ admits an exponential map; and that to each $X,Y\in \mg$ and $\vv\in \SEM$, there exists some $C\geq 0$ with
\begin{align}
\label{lklkasaaxyxy}
	\vvv(\com{X}^n(Y))\leq C^n \qquad\quad\forall\: n\geq 1.
\end{align}
Then, Corollary \ref{fiofiofdoifdio} shows that 
$G$ is quasi constricted in the sense of Sect.\ 8.3 in \cite{RGM}; so that Proposition 8 in \cite{RGM} also holds if there ``quasi constricted'' is replaced by Condition \eqref{lklkasaaxyxy}.\hspace*{\fill}$\ddagger$
\end{remark}
We furthermore obtain from Lemma \ref{podfopdfdofpdfpofdpo} that
\begin{corollary}
\label{fdpopodpofd}
Let $Z_1,\dots,Z_d\in \mg$, $m_1,\dots,m_d\in \NN$, and $\phi_\elll\in \DIDE_{[r_\elll,r'_\elll]}$ for $\elll=1,\dots, d$; and define
\begin{align*}
	t_{\elll,n,p}:=r_\elll +  p/n\cdot |r'_\elll-r_\elll|\qquad\quad\forall\: \elll=1,\dots, d,\:\: n\geq 1,\:\:p=0,\dots,n. 
\end{align*} 
Moreover, assume that for $\elll=1,\dots,d$ and $n\geq 1$,     
we are given maps 
\begin{align*}
	\alpha_{\phi_\elll,n}[p]\colon [t_{\elll,n,p},t_{\elll,n,p+1}]\times \mg\rightarrow \mg\qquad\quad\text{for}\qquad\quad p=0,\dots,n-1,
\end{align*} 
such that the following two conditions are fulfilled:
\begingroup
\setlength{\leftmargini}{16pt}
{
\renewcommand{\theenumi}{{\alph{enumi}})} 
\renewcommand{\labelenumi}{\theenumi}
\begin{enumerate}
\item
\label{hjfd0}
	We have $\alpha_{\phi_\elll,n,Y}\equiv\{\alpha_{\phi_\elll,n}[p](\cdot,Y)\}_{0\leq p\leq n-1}\in \CP_0^1([r_\elll,r_\elll'],\mg)$  with $\ul{\alpha_{\phi_\elll,n,Y}}(r_u)=Y$ for each $Y\in \mg$.  In this case, we let 
	\begin{align}
	\label{fddfdfdffddffd}
		\ul{\alpha_{\phi_\elll,n}}(t,Y):=\ul{\alpha_{\phi_\elll,n,Y}}(t)\qquad\quad\forall\: Y\in \mg,\:\: t\in [r_\elll,r'_\elll].
	\end{align}	
\item
\label{hjfd2}
We have $\{\alpha_{\phi_\elll,n,Y_n}\}_{n\in \NN}\slim{\phi_\elll} 0$  for each converging sequence $\mg\supseteq \{Y_n\}_{n\in \NN}\rightarrow Y\in \mg$.  
\end{enumerate}}
\endgroup
\noindent
Then, for each $Y\in \mg$, and $t_\elll\in [r_\elll,r'_\elll]$ for $\elll=1,\dots, d$, we have
\begin{align*}
	\textstyle\big(\com{Z_1}^{m_1}\cp\Add^{t_1}_{\phi_1}\cp  {\dots}\:\cp&\textstyle\:  \com{Z_d}^{m_d} \cp\Add^{t_d}_{\phi_d}\big)(Y)\\[1pt]
	&\textstyle=\lim_n \big(\com{Z_1}^{m_1}\cp\ul{\alpha_{\phi_1,n}}(t_1,\cdot)\cp {\dots}\cp \com{Z_d}^{m_d} \cp \ul{\alpha_{\phi_d,n}}(t_d,\cdot)\big)(Y).
\end{align*}
\end{corollary}
\begin{proof}
The claim just follows by induction on $p$ from Lemma \ref{podfopdfdofpdfpofdpo}.  
\end{proof}
\subsection{Asymptotic Estimates}
\label{pogfpogfa}
We now use Corollary \ref{fdpopodpofd} to estimate terms of the form \eqref{ndsnsnmsdnmdsnm} for the case that $\mg$ is  asymptotic estimate. Lemma \ref{podspods} is proven under milder presumptions, and will be applied in Sect.\ \ref{podspodspodspo} to the constricted case. 
In the following, we let $\exxp\colon \RR\rightarrow \RR_{>0}$ denote the exponential function, set $\euler:= \exp(1)$; and define 
\begin{align*}
	\textstyle\Lamb[X]_n\colon \RR\times \mg\rightarrow \mg,\qquad(t,Y)\mapsto \sum_{k=0}^n \frac{t^k}{k!}\cdot \com{X}^k(Y)
\end{align*}
for each $n\in \NN$, and $X\in \mg$. We fix $r<r'$, and define
\begin{align}
\label{fdfddfdddod}
	t_{n,p}:= r+ p/n\cdot |r'-r|\qquad\quad\forall\: n\geq 1,\:\: p=0,\dots,n. 
\end{align}
Let $\phi\in  \DIDE_{[r,r']}$ and $n\geq 1$ be fixed.
\begingroup
\setlength{\leftmargini}{13pt}
\begin{itemize}
\item
We define $\alpha_{\phi,n}[p]\colon [t_{n,p},t_{n,p+1}]\times\mg\rightarrow \mg$ inductively by 
\begin{align*}
\textstyle	\alpha_{\phi,n}[p]&:=\Lamb[\phi(t_{n,p})]_n(\cdot-t_{n,p}, \alpha_{\phi,n}[p-1](t_{n,p},\cdot))\qquad\quad\forall\: 0\leq p\leq n-1,
\end{align*}
with $\alpha_{\phi,n}[-1](t_{n,0},\cdot)\equiv\id_\mg$.
\item
Then, for each $Y\in \mg$ and $0\leq p\leq n-1$, we have (compare to \eqref{lkfdlkfdlkfd})
\begin{align*}
	\alpha_{\phi,n}[p](t_{p+1},Y)=\alpha_{\phi,n}[p+1](t_{p+1},Y)\qquad\quad\forall\: p=0,\dots,n-2.	
\end{align*}
Moreover, for $\tau\in [t_p,t_{p+1}]$ with $0\leq p\leq n-1$, we have 
\begin{align*}
	\textstyle\partial_t\he\alpha_{\phi,n}[p](\tau,Y)&-\bl \phi(\tau),\alpha_{\phi,n}[p](\tau,Y)\br\\[1pt]
	&=\textstyle\sum_{k=0}^{n-1}\textstyle\frac{(\tau-t_{n,p})^k}{k!}\cdot \com{\phi(t_{n,p})-\phi(\tau)}\cp\com{\phi(t_{n,p})}^{k}(\alpha_{\phi,n}[p-1](t_{n,p},Y))\\[1pt]
	&\quad\hspace{1pt}-\textstyle \frac{(\tau-t_{n,p})^{n}}{n!}\cdot \com{\phi(\tau)}\cp\com{\phi(t_{n,p})}^{n}(\alpha_{\phi,n}[p-1](t_{n,p},Y)).
\end{align*} 
\end{itemize}
\endgroup
\noindent
Evidently, then Condition \emph{\ref{hjfd0}} in Corollary \ref{fdpopodpofd} is fulfilled for $\phi_\elll\equiv \phi$  ($d\equiv 1$) there; and we define $\ul{\alpha_{\phi,n}}$ as in \eqref{fddfdfdffddffd}. 
Then, for $\pp\in \SEM$ given, we 
choose $\pp\leq \vv\in \SEM$ with 
\begin{align*}
	\ppp(\bl X_1,X_2\br)\leq \vvv(X_1)\cdot \vvv(X_2)\qquad\quad\forall\: X_1,X_2\in \mg;
\end{align*}
and obtain for $\tau\in [t_p,t_{p+1}]$ that 
\begin{align*}
	\ppp\big(\partial_t\he \alpha_{\phi,n}[p](\tau,Y)&-\bl \phi(\tau),\alpha_{\phi,n}[p](\tau,Y)\br\big)\\[2pt]
	&\leq\textstyle\vvv(\phi(t_{n,p})-\phi(\tau))\cdot\underbracket{\textstyle\sum_{k=0}^{n-1} \frac{|\tau-t_{n,p}|^k}{k!}\cdot \vvv(\com{\phi(t_{n,p})}^{k}(\alpha_{\phi,n}[p-1](t_{n,p},Y)))}_{{\bf A}}\\
	&\quad\he \textstyle + \vvv_\infty(\phi)\cdot\underbracket{\textstyle\frac{|\tau-t_{n,p}|^{n}}{n!}\cdot \vvv(\com{\phi(t_{n,p})}^{n}(\alpha_{\phi,n}[p-1](t_{n,p},Y)))}_{{\bf B}}
\end{align*}
holds. 
Assume now that for each $\vv\in \SEM$, there exist $\vv\leq \ww\in\SEM$ and $C_\vv\geq 0$ with
\begin{align}
\label{assssasaasxxxas}
	\vvv\cp \com{X_1}\cp {\dots}\cp \com{X_n}\leq  C_\vv^n\cdot \www\qquad\quad\forall\: X_1,\dots,X_n\in \im[\phi],\:\: n\geq 1.
\end{align}
Then, it is immediate from the definitions that 
\begin{align*}
	{\bf A}\textstyle\leq \exxp(|r'-r|\cdot C_\vv)\cdot \www(Y)\qquad\quad\:\:\text{and}\qquad\quad\:\:
	{\bf B}\leq\textstyle\frac{(|r'-r|\cdot C_\vv)^{n}}{n!}\cdot\exxp(|r'-r|\cdot C_\vv)\cdot\www(Y)
\end{align*}
holds; i.e., that additionally the condition \emph{\ref{hjfd2}} in Corollary \ref{fdpopodpofd} is fulfilled   
for $\phi_\elll\equiv \phi$ ($d\equiv 1$) there. 

We obtain the following two statements.
\begin{lemma}
\label{podspods}
Assume we are given $\mathrm{M}\subseteq \mg$, $\vv\leq \ww\in \SEM$, and $C_\vv\geq 0$, such that 
\begin{align}
\label{assssasaasasyyy}
	\vvv\cp \com{X_1}\cp {\dots}\cp \com{X_n}\leq  C_\vv^n\cdot \www\qquad\quad\forall\: X_1,\dots,X_n\in \mathrm{M},\:\: n\geq 1
\end{align}
holds. 
Then, for each $\phi\in \DP^0([r,r'],\mg)$ with $\im[\phi]\subseteq \mathrm{M}$, we have
\begin{align*}
	\textstyle\vvv\cp\Add_{\phi^{-1}}\leq \exxp(|r'-r|\cdot C_\vv)\cdot \www.
\end{align*}
\end{lemma}
\begin{proof}
Replacing $\mathrm{M}$ by $-\mathrm{M}\cup \mathrm{M}$ if necessary, we can assume that $-\mathrm{M}=\mathrm{M}$ holds.  
We let $\phi\in \DP^0([r,r'],\mg)$ with $\im[\phi]\subseteq \mathrm{M}$ be fixed, choose $\phi[0],\dots,\phi[m-1]$ as in \eqref{opopooppo}; and define $\phi_\elll\equiv\phi[\elll-1]$ for $1\leq \elll \leq d\equiv m$. Then, for $1\leq \elll\leq d$ and $t\in (t_{\elll-1},t_{\elll}]$ fixed, we define, cf.\ Example \ref{fdpofdopdpof}
\begin{align*}
	\psi_{\elll}:= \inver(\phi_\elll|_{[t_{\elll-1},t]}) \qquad\quad\:\:\text{as well as}\qquad\quad\:\:
	\psi_p:= \inverse{\phi_p}\qquad\forall\: p=1,\dots,\elll-1;
\end{align*}
and observe that $\im[\psi_1]\cup{\dots}\cup \im[\psi_u]\subseteq -\mathrm{M}=\mathrm{M}$ holds by \eqref{podspodspoaaa}, as well as
\begin{align*}
		\Add_{\phi^-}^t =     \Add^{t_1}_{\phi^-_1}\cp{\dots}\cp \Add^{t_{\elll-1}}_{\phi^-_{\elll-1}} \cp \Add^{t}_{\phi^-_{\elll}}=\Add_{\psi_1}\cp{\dots}\cp \Add_{\psi_{\elll-1}} \cp \Add_{\psi_{\elll}}
	\end{align*} 
	by \eqref{pfifpfpofdpofd}. 
	For $p=1,\dots,u$, we construct $\ul{\alpha_{\psi_p,n}}$ for $n\geq 1$ as above (for $\phi\equiv \psi_p$ there); and obtain from Corollary \ref{fdpopodpofd} (and \eqref{assssasaasxxxas}, \eqref{assssasaasasyyy}) that 
	\begin{align*}
		\textstyle\vvv(\Add^t_{\phi^{-1}}(Y))&\textstyle\hspace{3.4pt}=\hspace{3.2pt}\lim_n\vvv\big(\big(\ul{\alpha_{\psi_1,n}}(t_1,\cdot)\cp{\dots}\cp\ul{\alpha_{\psi_{\elll-1},n}}(t_{\elll-1},\cdot)\cp \ul{\alpha_{\psi_\elll,n}}(t,\cdot)\big)(Y)\big)\\
		&\stackrel{\eqref{assssasaasasyyy}}{\leq}\exxp(|r'-r|\cdot C_\vv)\cdot \ww(Y)
	\end{align*}
	holds for each $Y\in \mg$, which proves the claim.  
\end{proof}
\begin{lemma}
\label{fddfddffdfd}
Assume that $G$ is \emph{asymptotic estimate}; and let $\vv\leq\ww\in \SEM$ be as in \eqref{fdjfdlkfdfdlkmcx}. Then, for $\phi_1,\dots,\phi_d\in \DIDE_{[r,r']}$, $Z_1,\dots,Z_d,Y\in \mg$, $m_1,\dots,m_d\in \NN$, and $t\in [r,r']$, we have
\begin{align*}
	\textstyle\big(\vvv\cp \com{Z_1}^{m_1}\cp\Add^t_{\phi_1}\cp  {\dots}&\cp\textstyle\:\com{Z_d}^{m_d}\cp\Add^t_{\phi_d}\big)(Y)\\[1pt]
	&\textstyle\leq \exxp\big(\sum_{p=1}^d\int_r^t \www(\phi_p(s))\:\dd s\big) \cdot \www(Z_1)^{m_1}\cdot{\dots}\cdot\www(Z_d)^{m_d}\cdot \www(Y).
\end{align*}
\end{lemma}
\begin{proof}
For $\elll=1,\dots,d$, we construct $\ul{\alpha_{\phi_\elll,n}}$ for $n\geq 1$ as above (for $\phi\equiv \phi_\elll$ there); and let   
 $t_{n,p}$ be as in \eqref{fdfddfdddod} for $n\geq 1$ and $p=0,\dots,n$. 
Then, for
	$$
t \in 
\begin{cases}
[t_{n,p},t_{n,p+1})\quad &\text{for}\qquad 0 \leq p\leq n-2\\
[t_{n,n-1},t_{n,n}] &\text{for}\qquad p= n-1
\end{cases}
$$
and $\elll=1,\dots,d$, we define
	\begin{align*}
		\textstyle\mathrm{S}^t_{\ww,\elll,n}:= |t-t_{n,p}|\cdot \www(\phi_\elll(t_{n,p})) + \sum_{q=0}^{p-1} |t_{n,q+1}-t_{n,q}|\cdot \www(\phi_\elll(t_{n,q})).
	\end{align*}
Then, 
$\textstyle\lim_n \mathrm{S}^t_{\ww,\elll,n}=\int_r^t \www(\phi_\elll(s))\:\dd s$ holds; and, since \eqref{fdjfdlkfdfdlkmcx} implies  
	\begin{align*}
		\textstyle\big(\vvv\cp \com{Z_1}^{m_1}\cp \ul{\alpha_{\phi_1,n}}(t,\cdot)\cp{\dots}\:\cp  &\:\textstyle \com{Z_d}^{m_d}\cp\ul{\alpha_{\phi_d,n}}(t,\cdot)\big)(Y)\\[2pt]
	&\textstyle\leq \exxp\big(\sum_{\elll=1}^d{\mathrm{S}}^{t}_{\ww,\elll,n}\big) \cdot \www(Z_1)^{m_1}\cdot{\dots}\cdot\www(Z_d)^{m_d}\cdot \www(Y),
	\end{align*}
	 the claim is clear from Corollary \ref{fdpopodpofd}.	
\end{proof}

\section{Integrability}
\label{podspodspodspo}
In this brief section, we will apply Lemma \ref{podspods} to prove Proposition \ref{fdkjfdkjdfkjfdjfdlkjfdl}. 
For this, we recall the following definitions (cf.\ Sect.\ 7 in \cite{RGM}): 
\begingroup
\setlength{\leftmargini}{12pt}
\begin{itemize}
\item
A sequence $\{\phi_n\}_{n\in \NN}\subseteq\DP^0([0,1],\mg)$ is said to be {\bf tame} \defff for each $\vv\in \SEM$, there exists some $\vv\leq \ww\in \SEM$ with
\vspace{-10pt}
	\begin{align*}
		\vvv\cp \Ad_{[\innt_r^\bullet \phi_n]^{-1}}\leq \www\qquad\quad\forall\: n \in \NN.
	\end{align*}
	\vspace{-18pt}  
\item
$G$ is said to be {\bf $\boldsymbol{0}$-confined} \defff for each $\phi\in C^{0}([0,1],\mg)$, there exists a tame Cauchy sequence $\{\phi_n\}_{n\in \NN}\subseteq \DP^0([0,1],\mg)$ with $\{\phi_n\}_{n\in \NN}\rightarrow \phi$ uniformly.
\item
$G$ is said to be {\bf k-confined} for $k\in \NN_{\geq 1}\sqcup\{\lip,\infty\}$ \defff for each $\phi\in C^{k}([0,1],\mg)$, there exists a tame Mackey-Cauchy sequence $\{\phi_n\}_{n\in \NN}\subseteq \DP^0([0,1],\mg)$ with $\{\phi_n\}_{n\in \NN}\rightarrow \phi$ uniformly.
\end{itemize}
\endgroup
\begin{remark}
\label{kjdkjdskjdskjdiuiuiuiuew}
	The deviating definitions of {\rm k}-confinedness for $k\equiv 0$ and $k\in \NN_{\geq 1}\sqcup\{\lip,\infty\}$ are necessary; because, although each $\phi\in C^0([0,1],\mg)$ admits an approximating Cauchy sequence, there usually does not exist an approximating Mackey-Cauchy sequence.\footnote{Elsewise, each Mackey complete Hausdorff locally convex vector space would automatically be integral complete.} Specifically, in the proof of Theorem 3 in \cite{RGM} (cf.\ Theorem \ref{confev}), the integral of some given $\phi\in C^k([0,1],\mg)$ ($k\in \NN\sqcup \{\lip,\infty\}$) is defined pointwise, namely, as $\lim_n \innt_0^t \phi_n$ for each $t\in [0,1]$. Here, $\DP^0([0,1],\mg)\supseteq \{\phi_n\}_{n\in \NN}\rightarrow \phi$ denotes a uniformly converging tame
\begingroup
\setlength{\leftmargini}{12pt}
\begin{itemize}
\item
Cauchy sequence \hspace{40.8pt} for\:\: $k\equiv 0$.
\item
Mackey-Cauchy sequence\:\: for\:\: $k\in \NN_{\geq 1}\sqcup \{\lip,\infty\}$. 
\end{itemize}
\endgroup
\noindent	
The key point then is that (under the given presumptions) for each fixed $t\in [0,1]$, the sequence $\{\innt_0^t \phi_n\}_{n\in \NN}\subseteq G$ is a
\begingroup
\setlength{\leftmargini}{12pt}
\begin{itemize}
\item
Cauchy sequence for $k\equiv 0$; thus, converges if $G$ is sequentially complete.
\item
Mackey-Cauchy sequence for $k\in \NN_{\geq 1}\sqcup \{\lip,\infty\}$; thus, converges if $G$ is Mackey complete. 
\end{itemize}
\endgroup
\noindent
It then follows from local $\mu$-convexity that the so-defined map $\mu\colon [0,1]\ni t\mapsto \lim_n \innt_0^t \phi_n$ is continuous, and that $\{\innt_0^\bullet \phi_n\}_{n\in \NN}\rightarrow \mu$ converges uniformly. A standard argument involving \eqref{ixxxsdsdoisdiosda} (this formula also holds in the piecewise category, cf.\ Appendix E.1 in \cite{RGM}) 
then shows that $\mu$ is of class $C^1$, with $\Der(\mu)=\phi$.
\hspace*{\fill}$\ddagger$	
\end{remark}
We recall that $G$ is said to be constricted \defff \eqref{assssasaasasdsdsds} holds; and are ready for the 
\begin{proof}[Proof of Proposition \ref{fdkjfdkjdfkjfdjfdlkjfdl}]
Lemma \ref{jshjsdahjsd} provides us with the following statements:
\begingroup
\setlength{\leftmargini}{12pt}
\begin{itemize}
\item
For each $\phi\in C^0([0,1],\mg)$, there exists a Cauchy sequence $\{\phi_n\}_{n\in \NN}\subseteq \DP^\infty([0,1],\mg)$ with $\{\phi_n\}_{n\in \NN}\rightarrow \phi$ uniformly, such that $\bigcup_{n\in \NN}\im[\phi_n]$ is contained in a metrizable compact subset $\compacto\subseteq \mg$.
\item
For each $\phi\in C^\lip([0,1],\mg)$, there exists a Mackey-Cauchy sequence $\{\phi_n\}_{n\in \NN}\subseteq \DP^\infty([0,1],\mg)$  with $\{\phi_n\}_{n\in \NN}\rightarrow \phi$ uniformly, such that $\bigcup_{n\in \NN}\im[\phi_n]$ is contained in a metrizable compact subset $\compacto\subseteq \mg$.
\end{itemize}
\endgroup
\noindent 
Up to absorbing factors into seminorms, the claim thus follows from Lemma \ref{podspods} when applied to $\mathrm{M}\equiv \compacto$ as well as $\vv\leq \ww\in \SEM$ and $C_\vv\geq 0$ as in \eqref{assssasaasasdsdsds}. 
\end{proof}
\begin{remark}
As already mentioned in Sect.\ \ref{kjdsjdskjskjdskjsk},  
Proposition \ref{fdkjfdkjdfkjfdjfdlkjfdl} generalizes Proposition 5 in \cite{RGM}. Since $G$ is constricted if $\bl \cdot,\cdot\br$ is submultiplicative, Proposition \ref{fdkjfdkjdfkjfdjfdlkjfdl} also generalizes the first part of Proposition 6 in \cite{RGM}. 
\hspace*{\fill}$\ddagger$
\end{remark}

\section{Continuity of the Integral}
\label{podspodspodspo1}
In this section, we prove Theorem \ref{aoelsaoesalsaoelsa}. Our argumentation is based on the following proposition.
\begin{proposition}
\label{fdopodpofdsuepfdsfds}
Assume that $G$ is asymptotic estimate, and that $\evol_\infty$ 
is $C^\infty$-continuous. 
Then, for each  
$\pp\in \SEM$, there exist $\pp\leq\ww\in \SEM$ and $\dindq\in \NN$, such that
\begin{align*}
		\textstyle(\pp\cp\chart)\big( \innt \psii_m\cdot{\dots}\cdot \innt \psii_1\big)\leq 1
\end{align*} 
holds for all $\psii_1,\dots,\psii_m\in \DIDE_{[0,1/m]}^\infty$ ($m\geq 1$) with $\max(\www_\infty^\dindq(\psii_1),\dots, \www_\infty^\dindq(\psii_m))\leq 1$. 
\end{proposition}
The proof of Proposition \ref{fdopodpofdsuepfdsfds} is quite elaborate, and will be established step by step in the last two parts of this section. We now first use this proposition to prove Theorem \ref{aoelsaoesalsaoelsa}.
\subsection{The Main Result}
\label{opfdpodfpofdpofd}
We start with the following observation.
\begin{lemma}
\label{cxcxcxxccxcxxccxa}
For each $\mm\in \SEM$ and  $\phi\in \DIDE_{[0,1]}$, there exists some $m\geq 1$ with 
\begin{align*}
	\textstyle(\mm\cp\chart)\big(\innt_{(p-1)/m}^\bullet \phi|_{[(p-1)/m,p/m]}\big)\leq 1\qquad\quad\forall\: p=1,\dots,m.
\end{align*}
\end{lemma}
\begin{proof}
	We fix $\mu\colon I\rightarrow G$ ($I\subseteq \RR$ open with $[0,1]\subseteq I$) of class $C^1$ with $\Der(\mu)|_{[0,1]}=\phi$, choose $d>0$ such that $K_d\equiv [-d,1+d]\subseteq I$ holds, and define 
\begin{align*}
	\alpha\colon I\times I\ni(t,s)\mapsto \mu(t)\cdot \mu(s)^{-1}\in G.  
\end{align*}	
Since $[0,1]$ is compact, and since $\alpha$ is continuous with $\alpha(t,t)=e$ for each $t\in [0,1]$, to each open neighbourhood $U$ of $e$, there exists some $0<\delta_U\leq d$, such that 
	\begin{align*}
	\textstyle U\ni \alpha(t+s,t)=\textstyle \innt_t^{t+s}\Der(\mu)\qquad\quad\forall\: t\in [0,1],\:\:0\leq s\leq \delta_U
	\end{align*}
	holds; from which the claim is clear. 
\end{proof}
We more generally obtain the following statement.
\begin{lemma}
\label{cxcxcxxccxcxxccxaa}
Let $\qq\in \SEM$ be given. Then, there exists some $\qq\leq \mm\in \SEM$, such that to each $\phi\in \DIDE_{[0,1]}$, there exists some $m\geq 1$ with
\begin{align*}
	\textstyle(\qq\cp\chart)\big(\innt_{(p-1)/m}^\bullet \phi|_{[(p-1)/m,p/m]}\big)\leq 1/m\cdot \mmm_\infty(\phi)\qquad\quad\forall\: p=1,\dots,m.
\end{align*}
\end{lemma}
\begin{proof}
We choose $\qq\leq \mm\in \SEM$ as in \eqref{sadsndsdsnmdsds}; and for $\phi\in \DIDE_{[0,1]}$ fixed, we let $m\geq 1$ be as in Lemma \ref{cxcxcxxccxcxxccxa}. 
Then, the claim is clear from Lemma \ref{cxcxcxxccxcxxccxa} and \eqref{sadsndsdsnmdsdsa}.	
\end{proof}
We finally recall the following.
\begin{remark}
\label{dsdssds}
By Lemma 15 in \cite{RGM}, $\evol_0$ is $C^0$-continuous \deff it is $C^0$-continuous at zero. Thus, in order to prove Theorem \ref{aoelsaoesalsaoelsa}, it suffices to show that for each $\pp\in \SEM$, there exists some $\pp\leq \mm\in\SEM$, such that $(\pp\cp\chart)(\innt \phi)\leq 1$ holds for each $\phi\in \DIDE_{[0,1]}$ with $\mmm_\infty(\phi)\leq 1$. 
\hspace*{\fill}$\ddagger$
\end{remark}
We are ready for the proof of Theorem \ref{aoelsaoesalsaoelsa}.
\begin{proof}[Proof of Theorem \ref{aoelsaoesalsaoelsa}]
Let $\pp\in \SEM$ be fixed. We choose $\pp\leq \ww\in \SEM$ and $\dindq\in \NN$ as in Proposition \ref{fdopodpofdsuepfdsfds}, $\ww\leq \qq\in \SEM$ as in \eqref{omegaklla}, and $\qq\leq\mm\in \SEM$ as in Lemma \ref{cxcxcxxccxcxxccxaa}. By Remark \ref{dsdssds}, it suffices to show that
\begin{align*}
	\textstyle\mmm_\infty(\phi)\leq 1\quad\:\:\text{for}\quad\:\: \phi\in \DIDE_{[0,1]}\qquad\quad\Longrightarrow\qquad\quad (\pp\cp\chart)\big(\innt \phi\big)\leq 1. 
\end{align*}
Let thus $\phi\in \DIDE_{[0,1]}$ with $\mmm_\infty(\phi)\leq 1$ be given; and choose $m\geq 1$ as in Lemma \ref{cxcxcxxccxcxxccxaa}. Then, 
\begingroup
\setlength{\leftmargini}{12pt}
\begin{itemize}
\item
For $p=1,\dots,m$, we let 
\begin{align*}
	\textstyle X_p:=\chart\big(\innt \phi|_{[(p-1)/m,p/m]}\big)\qquad 
\quad \text{as well as}\qquad\quad 
\mu_p\colon [0,1/m]\ni t\mapsto \chartinv(t\cdot m\cdot X_p).
\end{align*} 
We obtain from Lemma \ref{cxcxcxxccxcxxccxaa} that $\qq(m\cdot X_p)\leq m\cdot \mm(X_p)\leq 1$ holds.
\item
We define $\psi_p:=\Der(\mu_p)$ for $p=1,\dots,m$. We obtain from \eqref{kldlkdldsl} and \eqref{omegaklla} that 
\begin{align*}
	\www\big(\psi_p^{(n)}(t)\big)= (\www\cp \dermapdiff[n])(t\cdot m\cdot X_p,\overbrace{m\cdot X_p,\dots,m\cdot X_p}^{n+1-\text{times}})\leq \qq(m\cdot X_p)^{n+1}\leq 1\qquad\quad\forall\: t\in [0,1/m]
\end{align*}
holds for $p=1,\dots,m$ and $0\leq n\leq \dindq$; thus, $\max(\www_\infty^\dindq(\psii_1),\dots, \www_\infty^\dindq(\psii_m))\leq 1$. 
\end{itemize}
\endgroup 
\noindent
By construction, as well as Part \ref{pogfpogf} of Proposition \ref{sddssddsxyxyxyxyyx} (last step), we have
\vspace{-2pt}
\begin{align*}
	\textstyle\innt\psi_m \cdot{\dots}\cdot \innt \psi_1=\mu_m(1/m)\cdot {\dots}\cdot \mu_1(1/m)= \innt \phi|_{[(m-1)/m,1]}\cdot {\dots}\cdot \innt \phi|_{[0,1/m]}\stackrel{\ref{pogfpogf}}{=}\innt \phi,
\end{align*}
so that the claim is clear from Proposition \ref{fdopodpofdsuepfdsfds}.
\end{proof}

\subsection{A Continuity Statement}
\label{sajjsaksasakjsajkas}
We now are going to prove Proposition \ref{fdopodpofdsuepfdsfds}. 
The key observation we will use is that    
$\phi_1,\dots,\phi_n\in \DIDE^\infty_{[0,1]}$ given, we inductively obtain from Part \ref{kdsasaasassaas} of Proposition \ref{sddssddsxyxyxyxyyx} that
\begin{align}
\label{lkdslkdslksdlkds}
\begin{split}
	\textstyle\innt \phi_n\cdot{\dots}\cdot \innt \phi_1
	&\textstyle\stackrel{\ref{kdsasaasassaas}}{=}\innt \phi_{n}\cdot {\dots}\cdot\innt \phi_{3}\cdot\innt (\phi_2+\Add_{\phi_2}(\phi_1))\\
&\stackrel{\ref{kdsasaasassaas}}{=}\boldsymbol{\dots}\\
	&\textstyle\stackrel{\ref{kdsasaasassaas}}{=}\innt \underbracket{\textstyle\phi_{n}+ \sum_{p=n-1}^{1} (\Add_{\phi_{n}}\cp{\dots}\cp\Add_{\phi_{p+1}})(\phi_p)}_{\hspace{5pt}\scalebox{1}{$\bchi{\phi_1,\dots,\phi_n}\in \DIDE^\infty_{[0,1]}$}}
	\end{split}
\end{align} 
holds. 
The strategy then is to estimate the higher derivatives of $\bchi{\phi_1,\dots,\phi_n}$ 
for the particular situation that each $\phi_p$ is of the form 
\begin{align}
\label{fdfdddfdfdfdddhjfdfdff}
	\textstyle\phi_p\colon [0,1]\ni t\mapsto 1/n\cdot\chi_p(t/n)
\end{align}
for certain $\chi_1,\dots,\chi_n\in \DIDE^\infty_{[0,1/n]}$, in which case Part \ref{subst} of Proposition \ref{sddssddsxyxyxyxyyx} yields
\begin{align}
\label{lkdslkdslksdlkdsa}
\begin{split}
	\textstyle\innt \chi_n\cdot{\dots}\cdot \innt \chi_1&\textstyle\hspace{3.3pt}\stackrel{\ref{subst}}{=}\hspace{3.3pt}\innt \phi_n\cdot{\dots}\cdot \innt \phi_1\textstyle\stackrel{\eqref{lkdslkdslksdlkds}}{=} \innt\underbracket{\bchi{\phi_1,\dots,\phi_n}}_{\hspace{-0.5pt}\scalebox{1}{$\bbchi$}}\!.
\end{split}
\end{align} 
We now will first clarify the particular form of the higher derivatives of the expression $\bchi{\phi_1,\dots,\phi_n}$ in \eqref{lkdslkdslksdlkds} for arbitrary $\phi_1,\dots,\phi_n\in \DIDE^\infty_{[0,1]}$.

\subsubsection{Some Combinatorics}
Let $\phi_1,\dots,\phi_n\in \DIDE^\infty_{[0,1]}$ and 
$\dind\in \NN$ be given. We consider a term of the form
\begin{align}
\label{lkdsldslksddsaa}
	\beta=\com{\psii[1]^{(\dind_1)}}^{m_1}\cp \Add_{\phi[1]}\cp{\dots}\cp \com{\psii[d]^{(\dind_{d})}}^{m_{d}}\cp\Add_{\phi[d]}\cp\psii[d+1]^{(\dind_{d+1})}
\end{align}
with $0\leq \dind_1,\dots,\dind_{d+1}\leq \dind$, $m_1,\dots,m_{d}\in \{0,1\}$, $\psii[1],\dots,\psii[d+1]\in \{\phi_1,\dots,\phi_n\}$, and $\phi[1],\dots,\phi[d]\in \{0,\phi_1,\dots,\phi_n\}$,  
such that\footnote{If $Z$ is a finite set, then $\# Z\in \NN$ denotes the cardinality of $Z$.} 
\begin{align}
\label{rerere1}
m_1\cdot (\dind_1+1) +{\dots}+m_{d}\cdot (\dind_{d}+1)+\dind_{d+1}&+1=\dind+1\tag{I},\\[2pt]
\label{rerere2}
	\#\{\iota_p\in \{1,\dots,d\}\:|\ \phi[\iota_p]\neq 0\}&\leq n.\tag{II}
\end{align}
We obtain by induction that
\begin{lemma}
\label{dfdfdfd}
For each $k\in \NN$, the $k$-th derivative of $\bchi{\phi_1,\dots,\phi_n}$  is a sum of at most $n\cdot (n+1)\cdot {\dots}\cdot (n+k)$ terms of the form \eqref{lkdsldslksddsaa} (fulfilling \eqref{rerere1} and \eqref{rerere2}) for $\dind\equiv k$ there. 
\end{lemma}
\begin{proof}
The claim is clear for $k=0$. We thus can assume that the claim holds for 
 some $k\geq 0$, i.e., that we have 
\vspace{-4pt}
\begin{align*}
	\textstyle\bchi{\phi_1,\dots,\phi_n}^{(k)}=\sum_{\elll=1}^{n\cdot(n+1)\cdot{\dots}\cdot(n+k)}\beta_\elll
\end{align*}
with each $\beta_\elll$ of the form \eqref{lkdsldslksddsaa} (fulfilling \eqref{rerere1} and \eqref{rerere2}) for $\dind\equiv k$ there. We let $\beta_\elll\equiv \beta$ be as in \eqref{lkdsldslksddsaa}, and define (for the second line observe $\Add_0=\id_\mg$)
\begin{align}
\label{kjsdkjdskjdskjdszdszudszuds}
\begin{split}
	\textstyle\Psi[p]:=&\textstyle \sum_{\ell=1}^{m_p}\: \com{\psi[p]^{(k_p)}}^{\ell-1}\cp        \com{\psi[p]^{(k_p+1)}}^1 \cp \com{\psi[p]^{(k_p)}}^{m_p-\ell}\\[3pt]
	\big(\hspace{-3pt}=&\textstyle \sum_{\ell=1}^{m_p}\: \com{\psi[p]^{(k_p)}}^{\ell-1}\cp \Add_0 \cp       \com{\psi[p]^{(k_p+1)}}^1 \cp \Add_0 \cp \com{\psi[p]^{(k_p)}}^{m_p-\ell} \big)
\end{split}
\end{align} 
for $p=1,\dots, d$. 
We inductively obtain from Lemma \ref{khdhjdhdkjkjdkjdjkd}, as well as the parts \ref{linear}, \ref{chainrule}, \ref{productrule} of Proposition \ref{iuiuiuiuuzuzuztztttrtrtr} that
{\allowdisplaybreaks
\begin{align*}
	\dot\beta_\elll&= \Psi[1]\cp \Add_{\phi[1]}\cp{\dots}\cp \com{\psii[d]^{(k_{d})}}^{m_{d}}\cp\Add_{\phi[d]}\cp \psii[d+1]^{(k_{d+1})}\\
	&\quad\he + \com{\psii[1]^{(k_1)}}^{m_1}\cp\partial_t\big(\Add_{\phi[1]}\cp{\dots}\cp \com{\psii[d]^{(k_{d})}}^{m_{d}}\cp\Add_{\phi[d]}\cp\psii[d+1]^{(k_{d+1})}\big)\\[4pt]
	&=\Psi[1]\cp \Add_{\phi[1]}\cp{\dots}\cp \com{\psii[d]^{(k_{d})}}^{m_{d}}\cp\Add_{\phi[d]}\cp\psii[d+1]^{(k_{d+1})}\\
	&\quad\he+ \com{\psii[1]^{(k_1)}}^{m_1}\cp \com{\phi[1]}\cp\Add_{\phi[1]}\cp{\dots}\cp \com{\psii[d]^{(k_{d})}}^{m_{d}}\cp\Add_{\phi[d]}\cp\psii[d+1]^{(k_{d+1})}\\
	&\quad\he+ \com{\psii[1]^{(k_1)}}^{m_1}\cp\Add_{\phi[1]}\cp \partial_t\big(\com{\psii[2]^{(k_2)}}^{m_2}\cp\Add_{\phi[2]}\cp {\dots} \:\cp \\
	&\textstyle\hspace{147.2pt}\com{\psii[d]^{(k_{d})}}^{m_{d}}  \cp \Add_{\phi[d]}\cp\psii[d+1]^{(k_{d+1})}\big)\\[2pt]
	&\textstyle={\boldsymbol{\dots}}\\[7pt]
	&\textstyle= \sum_{p=1}^{d} \com{\psii[1]^{(k_1)}}^{m_1}\cp \Add_{\phi[1]}\cp{\dots}\cp\Psi[p]\cp \Add_{\phi[p]} \\[1pt]
	&\hspace{162pt}\cp{\dots}\cp \com{\psii[d]^{(k_{d})}}^{m_{d}} \cp\Add_{\phi[d]}\cp \psii[d+1]^{(k_{d+1})}\\[2pt]
	&\hspace{51.5pt} + \com{\psii[1]^{(k_1)}}^{m_1}\cp \Add_{\phi[1]}\cp{\dots}\cp \com{\psii[d]^{(k_{d})}}^{m_{d}}\cp\Add_{\phi[d]}\cp \psii[d+1]^{(k_{d+1}+1)}\\[-9pt]
	&\hspace{12pt}\underbracket{\hspace{364pt}}_{\hspace{-8.8pt}\scalebox{1}{\bf{A}}}\\[7pt]
	&\quad\:\he\pl\he\textstyle \sum_{p=1}^d \com{\psii[1]^{(k_1)}}^{m_1}\cp \Add_{\phi[1]}\cp{\dots}\cp\com{\psii[p]^{(k_p)}}^{m_p}\cp\com{\phi[p]}\cp \Add_{\phi[p]}\\
	&\hspace{151.4pt}\cp{\dots}\cp \com{\psii[d]^{(k_{d})}}^{m_{d}}\cp \Add_{\phi[d]}\cp \psii[d+1]^{(k_{d+1})}\\[-9pt]
	&\hspace{25pt}\underbracket{\hspace{329.8pt}}_{\hspace{-0.5pt}\scalebox{0.9}{\bf{B}}}
\end{align*}}
\vspace{-5pt}

\noindent
holds. 
Evaluating $\Psi[p]$ for $p=1,\dots,d$ (second line in \eqref{kjsdkjdskjdskjdszdszudszuds}) and relabeling suitably, 
 we see that $\bf A$ decomposes into at most $k+1$ (use \eqref{rerere1}) summands that are of the form \eqref{lkdsldslksddsaa} (and fulfill \eqref{rerere1} and \eqref{rerere2}) for $\dind\equiv k+1$ there. 
 Moreover, by \eqref{rerere2} there occur at most $n$ non-zero summands in $\bf B$. Each of them can be brought into the form \eqref{lkdsldslksddsaa} (fulfilling \eqref{rerere1} and \eqref{rerere2}) for $\dind\equiv k+1$, by inserting an identity ($\Add_0\equiv\id_\mg$) in front of each $\com{\phi[p]}\equiv \com{\phi[p]^{(0)}}^1$, and then relabeling suitably. 
Consequently, each $\dot\beta_\elll$ can be expressed as a sum of less than $n+k+1$ summands of the form \eqref{lkdsldslksddsaa} (fulfilling \eqref{rerere1} and \eqref{rerere2}) for $\dind\equiv k+1$, from which the claim is clear. 
\end{proof}

\subsubsection{Proof of Proposition \ref{fdopodpofdsuepfdsfds}}
For $\chi_1,\dots,\chi_n\in \DIDE^\infty_{[0,1/n]}$ given, we define $\phi_1,\dots,\phi_n$ as in \eqref{fdfdddfdfdfdddhjfdfdff}; and let $\bbchi\equiv \bchi{\phi_1,\dots,\phi_n}$ be as in \eqref{lkdslkdslksdlkdsa}. Then, we obtain from Lemma \ref{fddfddffdfd} and Lemma \ref{dfdfdfd} that
\begin{lemma}
\label{kllklksalksalksalksa}
Assume that $G$ is asymptotic estimate; and let $\vv\leq \ww\in \SEM$ be as in  \eqref{fdjfdlkfdfdlkmcx}. Then, 
\begin{align*}
	\textstyle\vvv^\dindq_\infty(\bbchi)\leq \euler\cdot  \frac{(n+1)\cdot{\dots}\cdot(n+\dindq)}{n^{\dindq}}\qquad\quad\forall\:\dindq\in \NN
\end{align*}
holds for all $\chi_1,\dots,\chi_n\in \DIDE^\infty_{[0,1/n]}$ with 
$\max(\www_\infty^\dindq(\chi_1),\dots, \www_\infty^\dindq(\chi_n))\leq 1$. 
\end{lemma}
\begin{proof}
Let $0\leq k\leq \dindq$ be fixed. By Lemma \ref{dfdfdfd}, we have
\begin{align}
\label{lfdlkfdlkdlkd}
	\textstyle\bbchi^{(k)}=\sum_{\elll=1}^{n\cdot(n+1)\cdot{\dots}\cdot(n+k)}\beta_\elll
\end{align}
with $\beta_\elll$ of the form \eqref{lkdsldslksddsaa} (fulfilling \eqref{rerere1} and \eqref{rerere2}) for $\dind\equiv k$ there. Then, for  $\beta_\elll\equiv \beta$ as in \eqref{lkdsldslksddsaa}, we obtain from Lemma \ref{fddfddffdfd} that 
\begin{align}
\label{nmfdnmfdnmfdnmdf}
	\vvv_\infty(\beta_\elll)&\textstyle\leq \exxp\big(\underbrace{\textstyle\sum_{p=1}^d\int_0^1 \www(\phi[p](s))\:\dd s}_{\scalebox{0.9}{\bf A}}\big) \cdot 
	\underbrace{\textstyle\prod_{p=1}^{d+1}\www_\infty\big(\psii[p]^{(k_p)}\big)^{m_p}}_{\scalebox{0.9}{\bf B}}
\end{align}
\vspace{-9pt}

\noindent
holds, with $m_{d+1}\equiv 1$. 
\begingroup
\setlength{\leftmargini}{12pt}
\begin{itemize}
\item
We obtain from \eqref{rerere2} that
\begin{align*}
	\textstyle{\bf A}&\textstyle\hspace{3.5pt}\leq\hspace{3.2pt}  n\cdot\max\big(\int_0^1 \www(\phi_1(s))\:\dd s,\dots,\int_0^1 \www(\phi_n(s))\:\dd s\big)\\
	&\textstyle\stackrel{\eqref{substitRI}}{=}n\cdot \max\big(\int_0^{1/n} \www(\chi_1(s))\:\dd s,\dots,\int_0^{1/n} \www(\chi_n(s))\:\dd s\big) \\
	&\textstyle\hspace{3.5pt}\leq\hspace{3.2pt} \max(\www_\infty(\chi_1),\dots,\www_\infty(\chi_n))\leq 1.
\end{align*}
\item
Since we have $k_1,\dots,k_{d+1}\leq k\leq \dindq$, as well as $\max(\www_\infty^\dindq(\chi_1),\dots, \www_\infty^\dindq(\chi_n))\leq 1$, we obtain from \eqref{fdfdddfdfdfdddhjfdfdff} and Part \ref{chainrule} of Proposition \ref{iuiuiuiuuzuzuztztttrtrtr} that
\vspace{-3pt}
\begin{align*}
	\www_\infty\big(\psii[p]^{(k_p)}\big)&\hspace{8.8pt}\leq\hspace{8pt} \max\big(\www_\infty^{k_p}(\phi_1),{\dots},\www^{k_p}_\infty(\phi_n)\big)\\
	&\stackrel{\eqref{fdfdddfdfdfdddhjfdfdff}, \ref{chainrule}}{\leq} n^{-(k_p+1)}\cdot  \max\big(\www_\infty^{k_p}(\chi_1),{\dots},\www^{k_p}_\infty(\chi_n)\big)\\
	&\hspace{8.8pt}\leq\hspace{8pt} n^{-(k_p+1)}
\end{align*}
holds for $p=1,\dots,d+1$. We conclude from \eqref{rerere1} that
\begin{align*}
		{\bf B}\textstyle\leq n^{- [m_1\cdot (k_1+1) + {\dots} + m_d\cdot (k_d+1) + (k_{d+1}+1)]}= n^{-(k+1)}.
\end{align*} 
\end{itemize}
\endgroup 
\noindent
The triangle inequality applied to \eqref{lfdlkfdlkdlkd} gives 
\begin{align*}
	\textstyle\vvv_\infty(\bbchi^{(k)})\stackrel{\eqref{nmfdnmfdnmfdnmdf}}{\leq} \euler\cdot \frac{(n+1)\cdot{\dots}\cdot(n+k)}{n^{k}}\leq \euler\cdot \frac{(n+1)\cdot{\dots}\cdot(n+\dindq)}{n^{\dindq}}.
\end{align*}
Since this holds for each $0\leq k\leq \dindq$, the claim follows.
\end{proof}
We are ready for the proof of Proposition \ref{fdopodpofdsuepfdsfds}.
\begin{proof}[Proof of Proposition \ref{fdopodpofdsuepfdsfds}]
Let $\pp\in \SEM$ be fixed. Since $\evol_\infty$ 
is $C^\infty$-continuous, there exist $\pp\leq \vv\in \SEM$ and $\dindq\in \NN$, such that
\begin{align}
\label{posdopdsopspodspods}
	\textstyle\vvv_\infty^\dindq(\chi)\leq 3\quad\:\:\text{for}\:\:\quad \chi\in \DIDE_{[0,1]}^\infty\qquad\quad\Longrightarrow\qquad\quad (\pp\cp\chart)(\innt\chi)\leq 1.
\end{align} 
We choose $\vv\leq \ww$ as in \eqref{fdjfdlkfdfdlkmcx}, and let ($m\geq 1$)
\begin{align*}
	\psi_1,\dots,\psi_m\in \DIDE_{[0,1/m]}^\infty\qquad\quad\text{with}\qquad\quad\max(\mmm_\infty^\dindq(\psi_1),\dots ,\www_\infty^\dindq(\psi_m))\leq 1
\end{align*}
be given.  
For $q\geq 1$ fixed, we let $n:=  m\cdot q$, and define $\chi_1,\dots,\chi_{n}$ by
\begin{align*}
	\chi_{p\cdot q + \elll}\colon [0,1/(q\cdot m)]\ni t\mapsto\psi_p((p\cdot q+\elll)/(q\cdot m) + t ) \qquad\quad\forall\: p=0,\dots,m-1, \:\:\elll=0,\dots,q-1.
\end{align*}
By construction, we have $\max(\www_\infty^\dindq(\chi_1),\dots, \www_\infty^\dindq(\chi_{n}))\leq 1$, and the parts \ref{pogfpogf} and  \ref{subst} of Proposition \ref{sddssddsxyxyxyxyyx} show
\begin{align}
\label{podspodpoddsopods}
	\textstyle\innt \psi_m\cdot{\dots}\cdot \innt \psi_1 = \innt \chi_{n}\cdot{\dots}\cdot \innt \chi_1.
\end{align}
We let $\bbchi$ be as in \eqref{lkdslkdslksdlkdsa} for $\phi_1,\dots,\phi_n$ as in \eqref{fdfdddfdfdfdddhjfdfdff}; and obtain from Lemma \ref{kllklksalksalksalksa} that 
\begin{align*}
	\textstyle\vvv^\dindq_\infty(\bbchi)\leq \euler\cdot  \frac{(n+1)\cdot{\dots}\cdot(n+\dindq)}{n^{\dindq}}
\end{align*}
holds. 
Since the right hand side is bounded by $3$ for $q$ (thus, $n= q\cdot m$) suitably large, \eqref{posdopdsopspodspods} gives
\vspace{-3pt}
\begin{align*}
	\textstyle(\pp\cp\chart)\big(\innt\psi_m \cdot{\dots}\cdot \innt \psi_1\big) \stackrel{\eqref{podspodpoddsopods}}{=}\textstyle(\pp\cp\chart)\big(\innt \chi_{n}\cdot{\dots}\cdot \innt \chi_1\big)\stackrel{\eqref{lkdslkdslksdlkdsa}}{=}(\pp\cp\chart)\big(\innt\bbchi\big)\stackrel{\eqref{posdopdsopspodspods}}{\leq} 1,
\end{align*}
which proves the claim.
\end{proof}

\section*{Acknowledgements}
The author thanks Helge Gl\"ockner for general remarks on a draft of the present article. 
This work has been supported by the Alexander von Humboldt Foundation of Germany.

\addtocontents{toc}{\protect\setcounter{tocdepth}{0}}
\appendix

\section*{APPENDIX}
\section{Bastiani's Differential Calculus}
\label{Diffcalc}
In this Appendix, we recall the differential calculus from \cite{HA,HG,MIL,KHN}, cf.\ also Sect.\ 3.3.1 in \cite{RGM}. 
\vspace{6pt}

\noindent
Let $E$ and $F$ be Hausdorff locally convex vector spaces. 
A map $f\colon U\rightarrow E$, with $U\subseteq F$ open, is said to be  
differentiable \defff 
\begin{align*}
	\textstyle(D_v f)(x):=\lim_{t\rightarrow 0}1/t\cdot (f(x+t\cdot v)-f(x))\in E
\end{align*} 
exists for each $x\in U$, and $v\in F$. Moreover, $f$ is said to be $k$-times differentiable for $k\geq 1$ \defff 
	\begin{align*}
	D_{v_k,\dots,v_1}f\equiv D_{v_k}(D_{v_{k-1}}( {\dots} (D_{v_1}(f))\dots))\colon U\rightarrow E
\end{align*}
is defined for each $v_1,\dots,v_k\in F$ -- implicitly meaning that $f$ is $p$-times differentiable for each $1\leq p\leq k$. In this case, we define
\begin{align*}
	\dd^p_xf(v_1,\dots,v_p)\equiv \dd^p f(x,v_1,\dots,v_p):=D_{v_p,\dots,v_1}f(x)\qquad\quad\forall\: x\in U,\:\:v_1,\dots,v_p\in F
\end{align*} 	
for $p=1,\dots,k$; and let $\dd f\equiv \dd^1 f$, as well as $\dd_x f\equiv \dd^1_x f$ for each $x\in U$.  
Then, $f$ is said to be  
\begingroup
\setlength{\leftmargini}{11pt}
\begin{itemize}
\item
of class $C^0$ \defff it is continuous -- In this case, we let $\dd^0 f\equiv f$.
\item
of class $C^k$ for $k\geq 1$ \defff it is $k$-times differentiable, such that 
\begin{align*}
	\dd^pf\colon U\times F^p\rightarrow E,\qquad (x,v_1,\dots,v_p)\mapsto D_{v_p,\dots,v_1}f(x)
\end{align*} 
is continuous for each $p=0,\dots,k$.  
In this case, $\dd^p_x f$ is symmetric and $p$-multilinear for each $x\in U$ and $p=1,\dots,k$, cf.\ \cite{HG}.  
\item
of class $C^\infty$ \defff it is of class $C^k$ for each $k\in \NN$. 
\end{itemize}
\endgroup
\noindent
We have the following differentiation rules \cite{HG}.
\begin{custompr}{A.1}
\label{iuiuiuiuuzuzuztztttrtrtr}
The following assertions hold:
\begingroup
\setlength{\leftmargini}{18pt}
{
\renewcommand{\theenumi}{{\Alph{enumi}})} 
\renewcommand{\labelenumi}{\theenumi}
\begin{enumerate}
\item
\label{iterated}
A map $f\colon F\supseteq U\rightarrow E$ is of class $C^k$ for $k\geq 1$ \deff $\dd f$ is of class $C^{k-1}$ when considered as a map $F' \supseteq U' \rightarrow E$ for $F'\equiv F \times F$ and $U'\equiv U\times F$.
\item
\label{linear}
If $f\colon F\rightarrow E$ is linear and continuous, then $f$ is smooth with $\dd^1_xf=f$ for each $x\in F$ as well as $\dd^kf=0$ for each $k\geq 2$.  
\item
\label{chainrule}
	Assume that $f\colon F\supseteq U\rightarrow U'\subseteq F'$ and $f'\colon F'\supseteq U'\rightarrow U''\subseteq F''$ are of class $C^k$ for $k\geq 1$, for Hausdorff locally convex vector spaces $F,F',F''$. Then, $f'\cp f\colon U\rightarrow F''$ is of class $C^k$ with 
	\begin{align*}
		\dd_x(f'\cp f)=\dd_{f(x)}f'\cp \dd_x f\qquad\quad \forall\: x\in U.
	\end{align*}
\item
\label{productrule}
	Let $F_1,\dots,F_m,E$ be Hausdorff locally convex vector spaces, and $f\colon F_1\times {\dots} \times F_m\supseteq U\rightarrow E$ be of class $C^0$. Then, $f$ is of class $C^1$ \deff the ``partial derivatives''
	\begin{align*}
		\partial_p f \colon U\times F_p\ni((x_1,\dots,x_m),v_p)&\textstyle\mapsto \lim_{t\rightarrow 0} 1/t\cdot (f(x_1,\dots, x_p+t\cdot v_p,\dots,x_m)-f(x_1,\dots,x_m))
	\end{align*}
	exist in $E$ and are continuous, for $p=1,\dots,m$. In this case, we have
	\begin{align*}
		\textstyle\dd_{(x_1,\dots,x_m)} f(v_1,\dots,v_m)&\textstyle=\sum_{p=1}^m\partial_p f((x_1,\dots,x_m),v_p)\\
		&\textstyle\equiv \sum_{p=1}^m \dd f((x_1,\dots,x_m),(0,\dots,0, v_p,0,\dots,0))
	\end{align*}
	for each $(x_1,\dots,x_m)\in U$, and $v_p\in F_p$ for $p=1,\dots,m$.
\end{enumerate}}
\endgroup
\end{custompr}

\end{document}